\documentclass[11pt]{amsart}
\usepackage[margin=1in]{geometry}
\usepackage{hyperref}
\usepackage{graphicx}
\usepackage{amssymb}
\usepackage{epstopdf}
\usepackage{amsmath}
\usepackage[dvipsnames]{xcolor}
\usepackage{bm}
\usepackage{xcolor}
\usepackage{todonotes}
\usepackage{verbatim}
\usepackage{mathtools}
\usepackage{dsfont}
\newtheorem{theorem}{Theorem}[section]
\newtheorem{lemma}[theorem]{Lemma}
\newtheorem{corollary}[theorem]{Corollary}

\newtheorem{proposition}[theorem]{Proposition}
\newtheorem{definition}[theorem]{Definition}

\newtheorem{example}[theorem]{Example}
\newtheorem{observation}[theorem]{Observation}

\theoremstyle{remark}
\newtheorem{remark}[theorem]{Remark}

\newcommand{\N}{{\mathbb N}}
\newcommand{\Z}{{\mathbb Z}}

\renewcommand{\P}{\mathbb{P}}
\newcommand{\E}{\mathbb{E}}

\DeclareMathOperator{\LRM}{LRmax}
\DeclareMathOperator{\LRm}{LRmin}
\DeclareMathOperator{\RLM}{RLmax}
\DeclareMathOperator{\RLm}{RLmin}

\DeclareMathOperator{\pat}{pat}
\DeclareMathOperator{\coc}{c-occ}
\DeclareMathOperator{\occ}{\widetilde{occ}}
\DeclareMathOperator{\Perm}{Perm}
\DeclareMathOperator{\Leb}{Leb}
\newcommand{\Asi}{A_{\sigma,i}}
\newcommand{\leqsi}{\preccurlyeq_{\sigma,i}}
\newcommand{\Sr}{\SG_{\bullet}}

\newcommand{\Sri}{\tilde{\SG}_{\bullet}}
\newcommand{\SG}{\mathcal{S}}

\newcommand{\zz}{\bm{z}}

\newcommand{\bmX}{\bm{X}}
\newcommand\bmY{\bm{Y}}

\newcommand{\pu}{pos_U}
\newcommand{\ctu}{ct_U}
\newcommand{\pd}{pos_D}
\newcommand{\ctd}{ct_D}
\newcommand{\pr}{pos_R}
\newcommand{\ctr}{ct_R}
\newcommand{\pl}{pos_L}
\newcommand{\ctl}{ct_L}

\newcommand{\pdzr}{pos_D^{>z_0}}
\newcommand{\pdzl}{pos_D^{<z_0}}
\newcommand{\puzr}{pos_U^{>z_0}}

\title[Square Permutations are typically rectangular]{Square permutations are typically rectangular}

\author[J. Borga]{Jacopo Borga}
\author[E. Slivken]{Erik Slivken}

\address[J. Borga]{Institut für Mathematik, Universität Zürich, Winterthurerstr. 190, CH-8057 Zürich, Switzerland}\email{jacopo.borga@math.uzh.ch}

\address[E. Slivken]{Universit\'e Paris Dauphine, Place du Maréchal De Lattre De Tassigny, 75775, Paris, CEDEX 16, FRANCE}\email{eslivken@lpsm.paris}



\begin{document}
	
	\begin{abstract}
		We describe the limit (for two topologies) of large uniform random square permutations, \emph{i.e.,} permutations where every point is a record. The starting point for all our results is a sampling procedure for asymptotically uniform square permutations. Building on that, we first describe the global behavior by showing that these permutations have a permuton limit which can be described by a random rectangle.  We also explore fluctuations about this random rectangle, which we can describe through coupled Brownian motions.  Second, we consider the limiting behavior of the neighborhood of a point in the permutation through local limits. As a byproduct, we also determine the random limits of the proportion of occurrences (and consecutive occurrences) of any given pattern in a uniform random square permutation.
	\end{abstract}

\maketitle

\section{Introduction}

\subsection{Square permutations}
A \emph{record} of a permutation is a maximum or minimum, either from the left or the right.  For example, the point $(i,\sigma(i))$ for the permutation $\sigma$ is a left-to-right maximum if $\sigma(i)>\sigma(j)$, for all $j<i$.  We can think of the records as the external points of a permutation.  The points of a permutation that do not correspond to records are called \emph{internal points}.  \emph{Square permutations} are permutations where every point is a record (see Fig.~\ref{example_intro}).  We let $Sq(n)$ denote the set of square permutations of size $n$.

\begin{figure}[htbp]
	\begin{center}
		\includegraphics[scale=0.8]{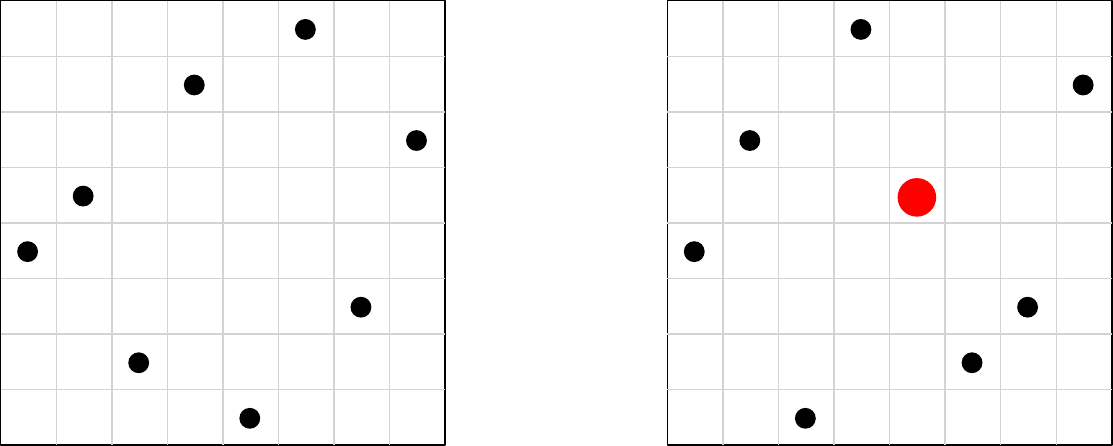}
		\caption{The diagram of two permutations $\sigma$ and $\pi$, \emph{i.e.,} the sets of points of the Cartesian plane at coordinates $(j,\sigma(j))$ and $(j,\pi(j)).$ The permutation on the left is a square permutation of size 8. The permutation on the right is not a square permutation since the bigger red dot is an internal point.}
		\label{example_intro}
	\end{center}
\end{figure}

Square permutations were first discussed in \cite{mansour_square} with connections to grid polygons, that is polygons whose vertices lie on a grid.  The authors computed the generating function for square permutations via the kernel method, and thus obtained the following enumeration.
\begin{theorem}[\cite{mansour_square}, Theorem 5.1]
\begin{equation}\label{sqenum}
|Sq(n)| = 2(n+2)4^{n-3} - 4(2n-5){ \binom{2n-6}{n-3}  }.	
\end{equation}
\end{theorem}

Square permutations were later discussed in \cite{duchi_square1} where the generating function was found by a more direct recursive approach, and again in \cite{duchi_square2} where a specific encoding (similar to that defined in Section \ref{sect:perm_to_anchored_seq} of this paper) of square permutations was used to make connections with convex permutominoes and their underlying generating functions.  The encoding leads to a linear time algorithm for generating a square permutation uniformly at random.  This encoding hints at the underlying structure of the square permutations, but much more is required to give the full geometric description of the limiting objects associated to square permutations.  Square permutations also appear in \cite{ALBERT2011715}, from a pattern-avoidance perspective.  The authors of \cite{ALBERT2011715} show that square permutations are precisely the permutations that avoid all 16 patterns of size five with an internal point.  They give an alternative approach to the enumeration of this class of permutation based on a context-free language used to describe the class.  In this paper they refer to square permutations as \emph{convex permutations}\footnote{In some sense this may be a better name for this class of permutations as the term `square permutations' occurs in a completely different context in \cite{square_original}.  Our introduction to these objects was from \cite{duchi_square1}, so we will stick with `square permutations'... for now.}.  

\subsection{Limiting shape of random permutations in permutation classes}
We say a permutation $\sigma$ avoids the pattern $\pi$ if no subsequence of $\sigma$ has the same relative order as $\pi$.  For a set of patterns $\mathcal{B}$ we say $\sigma$ avoids $\mathcal{B}$ if it avoids every pattern in $\mathcal{B}$.  Families of permutations that can be defined by pattern avoidance are called permutation classes.  Permutation classes have been widely studied (see \cite{bona, kit, vatter2014permutation} for a good introduction).  Typically the first question asked is enumerative: how many permutations of size $n$ are in a particular class?    

Recently, a probabilistic approach to the study of permutation classes has become quite popular.  A series of papers have taken this approach, exploring ideas like large deviations for permutations \cite{kenyon2015permutations, madras2010random, madras_monotone, mp}, path scaling limits \cite{hoffman2017pattern,hrs3}, and distributions of statistics associated with a class \cite{HRS1,HRS2,Ja14,Ja_multiple,Ja321,mansour_yildirim,mrs}.  

Another recent probabilistic framework is the theory of permutons introduced in \cite{hoppen2013limits}.  A permuton is a probability measure on the square $[0,1]^2$ with uniform marginals. Every permutation can be associated with the permuton induced by the sum of area measures on points of the permutation scaled to fit within $[0,1]^2$ (see Section \ref{sect:glob_beha} for a precise definition).  It is an exercise to show that permutations sampled uniformly at random on the whole symmetric group have a permuton limit given by Lebesgue measure on $[0,1]^2$.  The Mallows model is an example of non-uniform measure on permutations that also has deterministic permuton limit \cite{starr2009thermodynamic}.  

For permutations avoiding patterns of size three, or permutations avoiding longer monotone patterns, or permutations in some monotone grid classes, sampled uniformly at random, the permuton limits are also deterministic \cite{bevan2015growth,hoffman2017pattern,hrs3,madras_monotone}. 
More recently \emph{random} permuton limits were found in \cite{bassino2017universal, bassino2018brownian}. In these articles the \emph{Brownian separable permuton} (and modifications of it) was introduced and then a nice description of it in terms of the Brownian excursion was given in \cite{maazoun}. This random limiting object is also considered in \cite{bassino2019scaling,borga2018localsubclose}.  These were the first, and to best of our knowledge only, examples of (non-trivial) random permuton limits known.  

One of the results in the current paper is that square permutations have a \emph{random} permuton limit, though qualitatively quite different from the previously known random limits.

\subsection{Main tool: sampling asymptotically uniform square permutations}
\label{subsect:sampling}
The starting point for all our results is the sampling procedure described in this section. 

Inspired by the approach in \cite{hrs3}, we define a projection map from a square permutation $\sigma$ to the set of anchored pairs of sequences of labels, \emph{i.e.,} triplets $(X,Y,z_0)\in\{U,D\}^n\times\{L,R\}^n\times [n]$. For every square permutation $\sigma,$ the labels of $(X,Y)$ are determined by the record types. The sequence $X$ records if a point is a maximum ($U$) or a minimum ($D$) and the sequence $Y$ records if a point is a left-to-right record ($L$) or a right-to-left record ($R$); the anchor $z_0$ is the value $\sigma^{-1}(1)$.  Section \ref{sect:perm_to_anchored_seq} gives a precise definition and examples. 

This projection map is not surjective, but in Section \ref{sect:inverse_projection} we show that we can identify subsets of anchored pairs of sequences (called \emph{regular}) and of square permutations where the projection map is a  bijection. We then construct a simple algorithm to produce a square permutation from regular anchored pairs of sequences.  We show that asymptotically almost all square permutations can be constructed from regular anchored pairs of sequences, thus a permutation sampled uniformly from the set of regular anchored pairs of sequences will produce, asymptotically, a uniform square permutation.  

These regular anchored pairs of sequences are defined using a slight modification of the `Petrov conditions' \emph{i.e.,} technical conditions on the labels, found in \cite{hoffman2017pattern} and again in \cite{hrs3} (a uniform pair of sequences satisfies these conditions outside a set of exponentially small probability). We say that an anchored pair of sequences is \emph{regular} if it satisfies these Petrov conditions, and the anchored point is not too close to neither $1$ nor $n$.

\subsection{Main results: permuton limits, fluctuations and local limits}

In Section \ref{sect:glob_beha} we find the permuton limit of random square permutations created from regular anchored pairs of sequences sampled uniformly at random.  We show that for a large square permutation $\sigma$ that projects to a regular anchored pair of sequences, the permuton associated with $\sigma$ is close to a permuton given by a rectangle embedded in $[0,1]^2$ with sides of slope $\pm 1$ and bottom corner at $(\sigma^{-1}(1)/n,0).$ This allows us to show that the permuton limit of uniform square permutations is a rectangle embedded in $[0,1]^2$ with sides of slope $\pm 1$ and bottom corner at $(\bm z,0),$ where $\bm z$ is a uniform point in the interval $[0,1]$ (here and throughout the paper we denote random quantities using \textbf{bold} characters). See Fig.~\ref{typical} for some examples.

\begin{figure}
	\includegraphics[scale=.8]{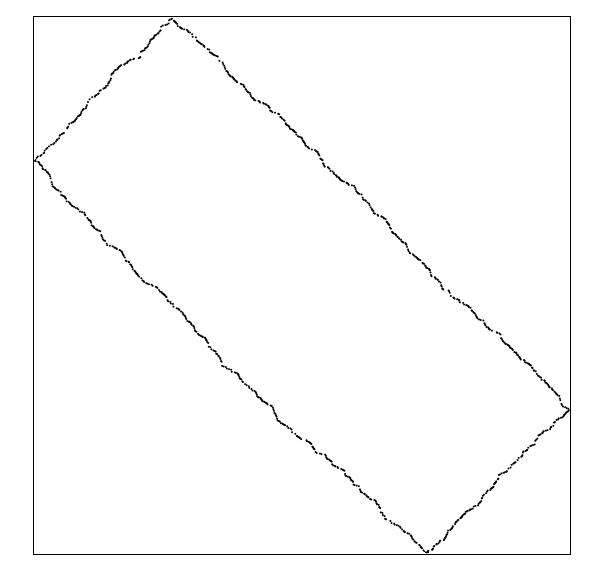}
	\hspace{.75cm}
	\includegraphics[scale=.8]{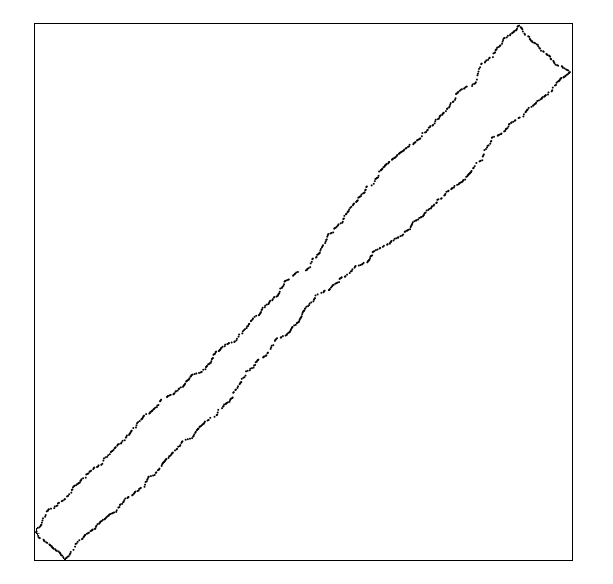}
	\caption{Two typical square permutations of size 1000. They have been obtained by sampling two uniform elements of $\{U,D\}^{1000}\times\{L,R\}^{1000}\times [1000]$ and then applying the algorithm that we developed to produce a square permutation from regular anchored pairs of sequences.}
	\label{typical}
\end{figure}

In Section \ref{sect:fluctuations} we show that the fluctuations about the lines of the rectangle of the permuton limit can be described by certain coupled Brownian motions.  The latter arise naturally from the projection map in Section \ref{sect:perm_to_anchored_seq}.  The technical difficulties in the proof come from the random size of the set of points for which we are measuring the fluctuations and the fact that these points are random in both coordinates.  The coupling between Brownian motions comes from the fact that the total number of labels of each type on a given interval (either horizontal or vertical) sums up to the size of the interval.    

We may also view the permutations locally as in \cite{borga2018local}. This local topology for permutations is the analogue of the celebrated Benjamini--Schramm convergence for graphs \cite{benjamini2001recurrence}, in the sense that we look at the neighborhood of a random element of the permutation. This viewpoint is explored in Section \ref{sect:local_beh}, where we prove that uniform square permutations locally converge (in the quenched sense) to a \emph{random} limiting object described in Section \ref{const_lim_obj}. This result answers an open question posed by the first author in \cite[Section 1.3]{borga2018local} of finding a natural (but non-trivial) model for which the quenched local limit is random. Indeed, in all the previously studied cases, that is, permutations avoiding some pattern of size three (see \cite{borga2018local}) and permutations in substitution closed-classes (see \cite{borga2018localsubclose}), the quenched local limit is deterministic.

Finally in Section \ref{sect:pattern_conv} we easily deduce the random limits of the proportion of occurrences (and consecutive occurrences) of any given pattern in a uniform random square permutation. This result is an immediate consequence of the permuton and quenched-local limits but it is worth mentioning. Indeed,
Janson \cite[Remark 1.1]{Ja_multiple} notices that, in some classes, we have concentration for the (non-consecutive) pattern occurrences around their mean, in others not. It would be interesting to understand in a more general setting when this concentration phenomenon does or does not occur for both pattern occurrences and consecutive pattern occurrences. Our results give a further example for when concentration does not occur.

\subsection{Possible future extensions}
The approach of assigning labels to the points of a permutation and then projecting these labels both horizontally and vertically is also used in the case of permutations avoiding monotone patterns \cite{hrs3}. However, in this model the total number of labels of each type in the horizontal and vertical projections must agree. Thus, to construct a permutation avoiding a monotone pattern by pairing up labels in an increasing fashion, the sequences must be conditioned to have the same total number for each label.  Surprisingly, this precise conditioning on the total number of each label is \emph{not} necessary in the case of square permutations, as long as the underlying anchored pair of sequences is regular.  This allows us to sample square permutations asymptotically uniformly at random in a much more straightforward manner. 

We highlight an interesting aspect of our approach introduced in Section \ref{subsect:sampling}, that is, sampling uniform permutations from a nicer subset (with asymptotically equal cardinality) of the considered set of permutations: to the best of our knowledge, this is the first technique that allows the study of permuton limits for uniform permutations in a fixed class that does not require the exact enumeration of the class (neither the explicit enumeration nor results on the behavior of the associated generating function). Although the enumeration for square permutations is known, this result is not needed (as noted in Remark \ref{noenumneeded}).    

Our approach seems to be quite generalizable.  The following are some ideas for future work. 

\begin{itemize}
	\item A natural question related to this article would be to understand how this model is affected by the introduction of a finite (or slowly increasing) number of internal points.  In \cite{DISANTO2011291}, the authors give a conjecture on the first order term of the number of permutations with exactly $k$ internal points.  We suspect a modified version of our approach will allow us to compute this first order term and describe how the addition of internal points changes the permuton limit.   
	
	{\em Note added in revision:} This problem has been investigated in \cite{borga2019almost}.     
	\item Monotone grid classes (introduced in \cite{huczynska2006grid} and then studied in many others works \cite{albert2013geometric,atkinson1999restricted,vatter2011partial,waton2007permutation}) seem to be a family of models where our approach fits well. We point out that some initial results on the shape of such permutations were given by Bevan in his Ph.D. thesis \cite{bevan2015growth}.  This approach suggests a way to give a description of the fluctuations and local limits of monotone grid classes.
	\item We also believe that this technique can give an alternative approach to the probabilistic study of $X$-permutations first considered in \cite{elizalde:the-x-class-and,waton2007permutation} and recently in \cite{bassino2019scaling}. In particular this technique could give nice additional results about the fluctuations similar to the ones explored in this paper. We finally recall that $X$-permutations are a particular instance of geometric grid classes (see \cite{albert2013geometric} for a nice introduction), and so also these families might be investigated in future projects.
\end{itemize}

\subsection*{Notation}

For this paper we view permutations of size $n$ as a set of points in $[n]\times[n]$ where each column and each row has exactly one point. We denote with $\mathcal{S}^n$ the set of permutations of size $n$ and with $\mathcal{S}$ the set of permutations of finite size.  Occasionally we may think of permutations as words of size $n$ or as bijections from $[n]\to[n]$.     The points associated with the four types of records are denoted by $\LRM, \LRm, \RLM,$ and $\RLm$.  These sets of points are not necessarily disjoint.  For any permutation $\sigma$, $(1,\sigma(1))$ is always in $\LRM\cap \LRm$.  Similar statements are true for $(n,\sigma(n))$, $(\sigma^{-1}(1),1)$ and $(\sigma^{-1}(n),n)$.  We call these the \emph{corners} of a permutation.  A permutation will have four corners unless it contains at least one of the points $(1,1)$, $(1,n)$, $(n,1)$, or $(n,n)$.  The permutation may also have points that are in $\LRM \cap \RLm$ or $\LRm \cap \RLM$.  By the pigeonhole principle the points in $\LRM\cap \RLm$ satisfy $i = \sigma(i)$ and the points in $\LRm \cap \RLM$ satisfy $i = n+1-\sigma(i).$     

If $x_1\dots x_n$ is a sequence of distinct numbers, let $\text{std}(x_1\dots x_n)$ be the unique permutation $\pi$ in $\SG^n$ that is in the same relative order as $x_1\dots x_n,$ \emph{i.e.}, $\pi(i)<\pi(j)$ if and only if $x_i<x_j.$
Given a permutation $\sigma\in\SG^n$ and a subset of indices $I\subseteq[n]$, let $\pat_I(\sigma)$ be the permutation induced by $(\sigma(i))_{i\in I},$ namely, $\pat_I(\sigma)\coloneqq\text{std}\big((\sigma(i))_{i\in I}\big).$
For example, if $\sigma=87532461$ and $I=\{2,4,7\}$ then $\pat_{\{2,4,7\}}(87532461)=\text{std}(736)=312$. 

Given two permutations, $\sigma\in\SG^n$ for some $n\in\N$ and $\pi\in\SG^k$ for some $k\leq n,$ we say that $\sigma$ contains $\pi$ as a \textit{pattern} if $\sigma$ has a subsequence of entries order-isomorphic to $\pi,$ that is, if there exists a \emph{subset} $I\subseteq[n]$ such that $\pat_I(\sigma)=\pi$.
Denoting $i_1, i_2, \dots, i_k$ the elements of $I$ in increasing order, the subsequence $\sigma(i_1) \sigma(i_2) \dots \sigma(i_k)$ is called an \emph{occurrence} of $\pi$ in $\sigma$. 
In addition, we say that $\sigma$ contains $\pi$ as a \textit{consecutive pattern} if $\sigma$ has a subsequence of adjacent entries order-isomorphic to $\pi$, 
that is, if there exists an \emph{interval} $I\subseteq[n]$ such that $\pat_I(\sigma)=\pi$. 
Using the same notation as above, $\sigma(i_1) \sigma(i_2) \dots \sigma(i_k)$ is then called a \emph{consecutive occurrence} of $\pi$ in $\sigma$. 

\begin{example} 
	The permutation $\sigma=1532467$ contains $1423$ as a pattern but not as a consecutive pattern and $321$ as consecutive pattern. Indeed $\pat_{\{1,2,3,5\}}(\sigma)=1423$ but no interval of indices of $\sigma$ induces the permutation $1423.$ Moreover, $\pat_{[2,4]}(\sigma)=\pat_{\{2,3,4\}}(\sigma)=321.$
\end{example}

We denote by $\occ(\pi,\sigma)$ the proportion of occurrences of a pattern $\pi$ (of size $k$) in $\sigma$ (of size $n$).
More formally
\[\occ(\pi,\sigma) := \frac{1}{\binom{n}{k}} \, \mathrm{card}\big\{I \subseteq [n] \text{ of cardinality }k \text{ such that } \pat_I(\sigma)=\pi\big\}.\]

We also denote by $\widetilde{\coc}(\pi,\sigma)$ the proportion of consecutive occurrences of a pattern $\pi$ in $\sigma$.
More precisely
\begin{equation*}
\widetilde{\coc}(\pi,\sigma)\coloneqq\frac{1}{n}\mathrm{card}\big\{I \subseteq [n] \text{ of cardinality }k \text{ such that } I\text{ is an interval and } \pat_I(\sigma)=\pi\big\}.
\end{equation*}   

\section{Projections for square permutations}
\label{sect:perm_to_anchored_seq}
We begin this section with the following key definition.
\begin{definition}
	A \emph{anchored pair of sequences} is a triplet $(X,Y,z_0)$, where $X\in\{U,D\}^n$, \break
	$Y\in\{L,R\}^n$ and $z_0 \in [n].$ We say that the pair $(X,Y)$ is anchored at $z_0$.
\end{definition}
Given a square permutation $\sigma\in Sq(n),$ we associate to it an anchored pair of sequences $(X,Y,z_0)$ in the following way (\emph{cf.}\ Fig.~\ref{Square_perm_sampling_example}). First let $X_1 = X_n = D$ and $Y_1 = Y_n = L$. Then, for all $i\in \{2,\cdots,n-1\},$  we set
\begin{itemize}
	\item  $X_i=D$ (resp.\ $X_i=U$) if $(i,\sigma(i))\in\LRm\cup\RLm$ (resp.\ $(i,\sigma(i))\in\LRM\cup\RLM$);
	\item  $Y_{i}=L$ (resp.\ $Y_{i}=R$) if $(\sigma^{-1}(i),i)\in\LRm\cup\LRM$ (resp.\ $(\sigma^{-1}(i),i)\in\RLm\cup\RLM$).
\end{itemize}
In the case that $(i,\sigma(i))\in\LRM\cap\RLm$ or $(i,\sigma(i))\in\LRm\cap\RLM$ we set $X_i=D$ and $Y_{\sigma(i)}=L.$
Finally, let $z_0 = \sigma^{-1}(1)$. Note that $X_{z_0}$ is always equal to $D$.

Intuitively, the sequence $X$ tracks if the points in the columns of the diagram of $\sigma$ are minima or a maxima and the sequence $Y$ tracks if the points in the rows are left or right records.

We denote with $\phi$ the map that associates to every square permutation the corresponding anchored pair of sequences, therefore
$$\phi:Sq(n)\to\{U,D\}^n\times\{L,R\}^n\times[n].$$

\begin{remark}
	This projection map is also used in \cite{duchi_square2}.  The author shows that $\phi$ is an injective map from $Sq(n)$ into the space of good anchored pairs of sequences.  
\end{remark}

\begin{figure}[htbp]
	\begin{center}
		\includegraphics[scale=0.5]{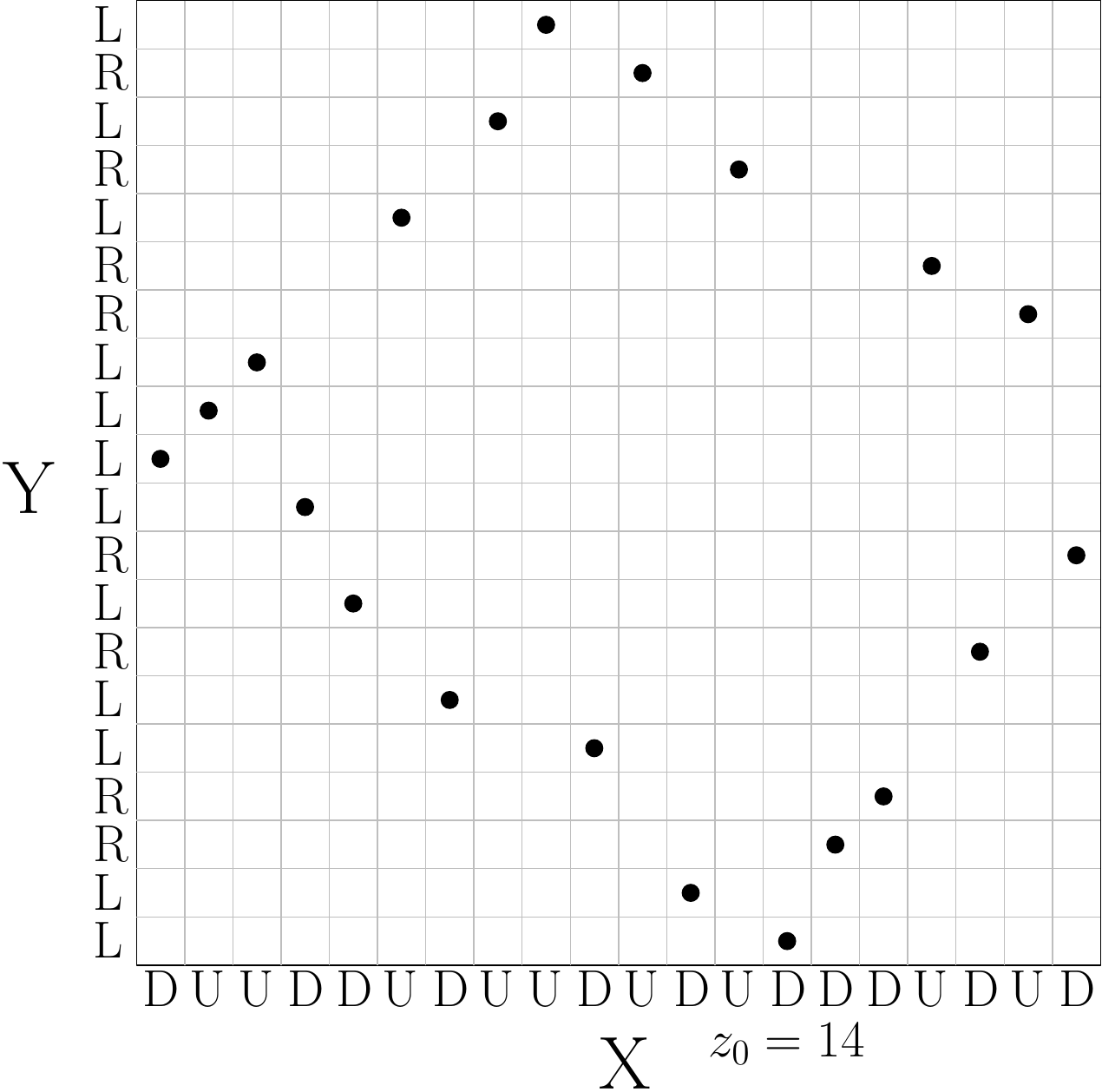}
		\caption{A square permutation $\sigma$ with the associated anchored pair of sequences $\phi(\sigma)=(X,Y,z_0).$  The sequence $X$ is reported under the diagram (read from left to right) of the permutation and the sequence $Y$ on the left (read from bottom to top).}
		\label{Square_perm_sampling_example}
	\end{center}
\end{figure}

We say that an anchored pair of sequences $(X,Y,z_0)$ of size $n$ is \emph{good} if $X_1 = X_n = X_{z_0}=D$ and $Y_1 = Y_n = L$. Note that $\phi(Sq(n))$ is contained in the set of good anchored pairs of size $n$. The total number of possible good anchored pairs $(X,Y,z_0)$ of size $n$ is
\begin{equation}\label{karl}
2^{n-2}(2\cdot 2^{n-2}+(n-2)\cdot 2^{n-3})= 2(n+2)4^{n-3}.	
\end{equation}

\section{Constructing permutations from anchored pairs of sequences}\label{sect:inverse_projection}

\subsection{Anchored pairs of sequences and Petrov conditions}
 
From a good anchored pair we wish to construct a square permutation.  For most good anchored pairs this will be possible, though we do need to proceed with caution.

Let $\ctd(i)$ denote the number of $D$s in $X$ up to (and including) position $i$.  Similarly define $\ctu(i)$, $\ctl(i)$ and $\ctr(i)$ for the number of $U$s in $X$ and the number of $L$s or $R$s in $Y$, respectively.  Let $\pd(i)$ denote the index of the $i$th $D$ in $X$ with $\pd(i) = n$ if there are fewer than $i$ indices labeled with $D$ in $X$.  Similarly define $\pu(i)$, $\pl(i)$ and $\pr(i)$ for the location of the indices of the other labels.  

The following are easy but useful properties of the  sequence $X$:
\begin{enumerate}
\item $\ctd(\pd(j)) = j$, for all $j\leq \ctd(n)$;
\item $\ctu(\pu(j)) = j$, for all $j\leq \ctu(n)$;  
\item If $X_j = D$ then $\pd(\ctd(j)) = j$;
\item If $X_j = U$ then $\pu(\ctu(j)) = j$; 
\item $\ctu(i) + \ctd(i) = i$, for all $i\in [n]$.
\end{enumerate}
Similar properties hold for $Y$ with the appropriate functions.

\begin{definition}[Petrov conditions]
Similar to Definition 2.3 in \cite{hoffman2017pattern}, we say that the label $D$ in $X$ satisfies the Petrov conditions if the following are true:

\begin{enumerate}
	\item $|\ctd(i) - \ctd(j) - \frac12(i-j) | < n^{.4}$, for all $|i-j| < n^{.6}$;
	\item $|\ctd(i) - \ctd(j) - \frac12(i-j) |	 < \frac{1}{2}|i-j|^{.6}$, for all $|i-j|> n^{.3}$;
	\item $|\pd(i) - \pd(j) - 2(i-j)| < n^{.4}$, for all $|i-j| < n^{.6}$ and $i,j \leq \ctd(n)$;
	\item $|\pd(i) - \pd(j) - 2(i-j)| < 2|i-j|^{.6}$,	 for all $|i-j|> n^{.3}$ and $i,j \leq \ctd(n)$.
\end{enumerate}

In particular, in the above inequalities we allow $j$ to be equal to 0 (defining $\ctd(0)\coloneqq 0$ and $\pd(j)\coloneqq 0$) obtaining that

\begin{enumerate}
	\item[(5)] $|\ctd(i) - \frac12 i | < n^{.6}$, for all $i\leq n$;
	\item[(6)] $|\pd(i)- 2i| < 2n^{.6}$, for all $i\leq  \ctd(n)$.
\end{enumerate}

\end{definition}

A similar definition holds for the label $U$ in $X$ and the labels $L$ and $R$ in $Y$ for the functions $\ctu,\ctl,\ctr,$ and $\pu, \pl, \pr$.  We say the Petrov conditions hold for the pair of sequences $X$ and $Y$ if the Petrov conditions hold for all the labels of $X$ and $Y$. We state a technical result.

\begin{lemma}\label{petrov_often}
If $X \in \{U,D\}^n$ satisfies the Petrov conditions then, for all $i\leq n$, 
$$i - \pd(\ctd(i)) \leq n^{.3}.$$
\end{lemma}

\begin{proof}
Fix $i\leq n.$ By contradiction suppose that $i-\pd(\ctd(i))>n^{.3},$ then using the second Petrov condition we have
	\begin{equation}
	\label{eq:contradiction}
	\left|\ctu(i)-\ctu(\pd(\ctd(i)))-\frac{1}{2}(i-\pd(\ctd(i)))\right|<\frac 12 \left|i-\pd(\ctd(i))\right|^{.6}.
	\end{equation}
	Noting that $\pd(\ctd(i))$ is the index of the right-most $D$ before the $i$-th position in $X,$ we have that $$\ctu(i)-\ctu(\pd(\ctd(i)))=i-\pd(\ctd(i)),$$
	and so we obtain a contradiction in the above Equation (\ref{eq:contradiction}).
\end{proof}

If the Petrov conditions are satisfied then the functions $\pu(i)$ and $\pd(i)$ are closely related in the following sense:  we define
\begin{equation*}
\begin{split}
&e(i):=\pu(i)-2i,\quad\text{for all}\quad i \leq \ctu(n),\\
&s(i):=\pd(i)-2i+e(i),\quad\text{for all}\quad i \leq \min\{\ctu(n),\ctd(n)\},
\end{split}
\end{equation*}
in order to have the expressions
\begin{equation*}
\pu(i)=2i+e(i)\quad\text{and}\quad \pd(i)=2i-e(i)+s(i).
\end{equation*}

\begin{lemma}
	\label{lemm: technical_lemma_for_fluctuation}
	If $X$ is a sequence that satisfies the Petrov conditions then it holds that
	\[|s(i)|<10n^{.4},\quad\text{for all}\quad i\leq \min\{\ctu(n),\ctd(n)\}.\]
\end{lemma}

\begin{proof}
	Fix $i\leq \min\{\ctu(n),\ctd(n)\}.$ We first note that 
	\begin{equation*}
	\ctd(\pu(i))=\pu(i)-\ctu(\pu(i))=\pu(i)-i.
	\end{equation*}
	Therefore, applying the operator $\pd(\cdot)$ on both sides of the previous equation, we obtain 
	\begin{equation}
	\label{eq:cond_1}
	\pd(\ctd(\pu(i)))=\pd(\pu(i)-i).
	\end{equation}
	From (\ref{eq:cond_1}) and Lemma \ref{petrov_often}, we obtain
	\begin{equation}
	\label{eq:cond_3}
	|\pd(\pu(i)-i)-\pu(i)|\leq n^{.3}.
	\end{equation}
	
	Now, using the Petrov conditions, we have that $|\pu(i)-2i|<2n^{.6}$ and that
	\begin{equation}
	\label{eq:cond_4}
	|\pd(\pu(i)-i)-\pd(i)-2(\pu(i)-2i)|<\max\{n^{.4},2(2n^{.6})^{.6}\}<4n^{.4}.
	\end{equation}
	Since $s(i):=\pd(i)+\pu(i)-4i$, we can conclude, using (\ref{eq:cond_3}) and (\ref{eq:cond_4}), that 
	\begin{equation*}
	|s(i)|=|\pu(i)-\pd(i)-2(\pu(i)-2i)|<4n^{.4}+n^{.3}<10n^{.4}.\qedhere
	\end{equation*}
\end{proof}

We let $\Omega_{n}$ denote the space of good anchored pairs of sequences such that both $X$ and $Y$ satisfy the Petrov conditions and $ n^{.9} \leq z_0 \leq n - n^{.9}.$  We will refer to these as \emph{regular} anchored pairs of sequences.

\begin{lemma}\label{omega_size}
Let $X$, $Y$ and $z_0$ be chosen independently and uniformly at random from $\{U,D\}^n$, $\{L,R\}^n$ and $\{1, \cdots , n\}$ respectively.  Then $\P((X,Y,z_0) \in \Omega_n)=1-o(1).$
\end{lemma}

\begin{proof}
Using standard Petrov style moderate deviations \cite{Petrov1} for some $c>0$, both $X$ and $Y$ satisfy the Petrov conditions with probability at least $1-\exp(-n^c)$ and $n^{.9} < z_0 < n - n^{.9}$ with probability at least $(1-2n^{-.1}).$  The lemma follows from the independence of $X$, $Y$ and $z_0$.
\end{proof}

\subsection{From regular anchored pairs of sequences to square permutations}
We wish to define a map $\rho: \Omega_n \to Sq(n)$ by constructing a unique matching between the labels of $X$ and the labels of $Y$.  The matching will depend on parameter $z_0$.  Once this map is properly defined we will show that the composition $\phi\circ \rho$ acts as the identity on $\Omega_n$ (see Lemma \ref{lem:map_welldef}).  

This construction may not be well defined on every good anchored pair of sequences in $\{U,D\}^n\times \{L,R\}^n \times [n]$, but will be well-defined if we restrict to $\Omega_n$.

First define the following values (whose role will be clearer later) with respect to $(X,Y,z_0) \in \Omega_n$:
\begin{itemize}
\item $z_1 = \pl(\ctd(z_0)),$
\item $z_2 = \pu( \ctl(n) - \ctd(z_0)),$
\item $z_3 = \pr( \ctd(n) - \ctd(z_0) ).$ 	
\end{itemize} 

The following lemma states a regularity property satisfied by the points $z_1,z_2$ and $z_3$.

\begin{lemma}\label{okz}
	Let $(X,Y,z_0) \in \Omega_n.$ Then $\max(|z_1-z_0|,|z_2-z_3|, |n-z_0-z_2|) < 10n^{.6}.$ 
	%Moreover, for $n$ big enough, $\frac{n^{.9}}{2}<z_i<n-\frac{n^{.9}}{2},$ for all $i=1,2,3$.
\end{lemma}

\begin{proof}	
	Note that by the Petrov conditions
	\begin{multline*}
	|z_1-z_0| 
	< |2\ctd(z_0) -z_0| + |\pl(\ctd(z_0)) - 2\ctd(z_0)| \\
	< 2|\ctd(z_0) - z_0/2| +2|\ctd(z_0)|^{.6}
	< 2n^{.6}+ 2n^{.6}=4n^{.6}.
	\end{multline*} 
	Similar arguments hold to show that $|z_2-z_3|$ and $|n-z_0-z_2|$ have the same upper bound. 
\end{proof}

Define the following index sets:
\begin{itemize}
\item $I_1 = \{1,\cdots, \ctd(z_0)\}$;
\item $I_2 = \{1,\cdots,\ctu(z_2)\}$;
\item $I_3 = \{\ctd(z_0) +1,\cdots, \ctd(n)\}$;
\item $I_4 = \{\ctu(z_2) +1, \cdots, \ctu(n)\}$.	
\end{itemize}
Using these index sets, we create four sets of points:
\begin{itemize}
\item $\Lambda_1 = \{(\pd(i),\pl(\ctd(z_0) + 1 -i))\}_{ i \in I_1}$;
\item $\Lambda_2 = \{(\pu(i),\pl(\ctd(z_0) +i))\}_{ i \in I_2}$;
\item $\Lambda_3 = \{(\pd(i),\pr(i-\ctd(z_0)))\}_{ i \in I_3}$; 
\item $\Lambda_4 = \{(\pu(i),\pr(n-\ctd(z_0)+1-i))\}_{ i \in I_4}$. 
\end{itemize}

\begin{figure}\includegraphics[scale=.5]{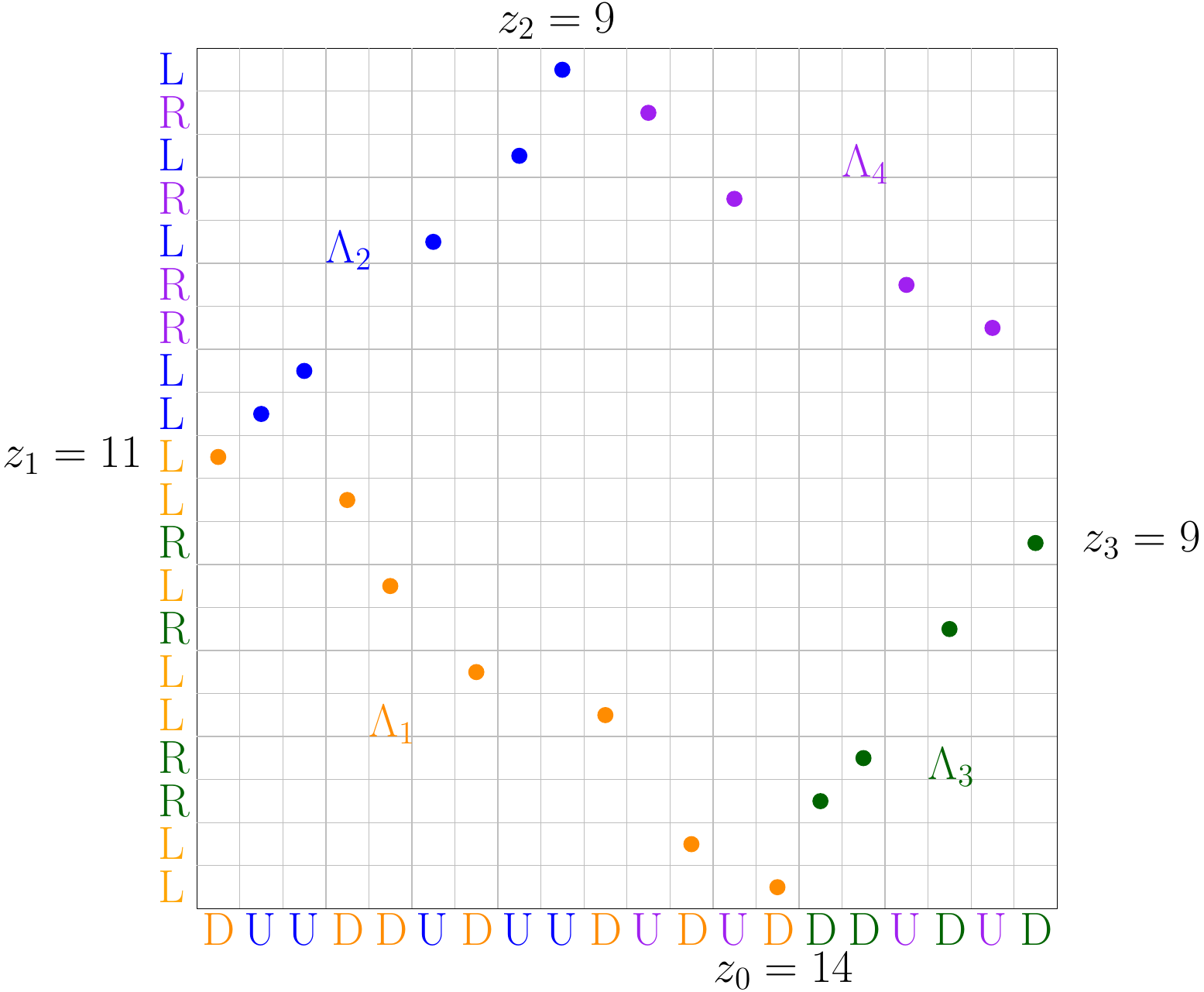}
	\caption{An example of square permutation obtained from a regular anchored pair of sequences $(X,Y,z_0)$. We color in orange the labels in the sequences $X$ and $Y$ corresponding to indexes of $I_1$ and we also color in orange the points of $\Lambda_1.$ Similarly, we color in blue the labels and the points corresponding to $I_2$ and $\Lambda_2,$ in green the ones corresponding to $I_3$ and $\Lambda_3$, and in purple the ones corresponding to $I_4$ and $\Lambda_4$.}
	\label{reconstruction}	
\end{figure}

First a few observations about each of the sequences $\Lambda_i$ (\emph{cf.}\ Fig.~\ref{reconstruction}):  
\begin{itemize}
	\item The first sequence, $\Lambda_1,$ is obtained by matching the first $\ctd(z_0)$ $D$s in $X$ with the first $\ctd(z_0)$ $L$s in $Y$ to create a decreasing sequence starting from $(1,z_1)$ and ending at $(z_0,1)$ (note that thanks to Petrov conditions, for $n$ big enough, $\ctd(z_0)$ is smaller than the total number of $L$s in $Y$ and so the operations is well-defined\footnote{We point out that the set $\Lambda_1$ is in full generality well-defined for all $n\geq1$, but in the case that $\ctd(z_0)>\ctl(n)$ then $|\Lambda_1|<|I_1|$ (because by definition  $\pl(i)=n$ for all  $i\geq\ctl(n)$). Similar observations will hold also for $\Lambda_2$, $\Lambda_3,$ and $\Lambda_4$.});
	\item the second sequence, $\Lambda_2,$ is obtained by matching the remaining $\ctl(n)-\ctd(z_0)$ $L$s in $Y$ with the first $\ctl(n)-\ctd(z_0)$ $U$s in $X$ (using Petrov conditions it easy to show that for $n$ big enough $\ctl(n) - \ctd(z_0)<\ctu(n)$) to create an increasing sequence starting from $(\pu(1),\pl(\ctd(z_0)+1)$ and ending at $(z_2,\pl(\ctd(z_0)+\ctu(z_2))).$
	Recalling that by definition $z_2=\pu( \ctl(n) - \ctd(z_0))$, we obtain that  $\pl(\ctd(z_0)+\ctu(z_2)) = \pl(\ctl(n)) = n$ since $Y_n = L$. Thus, $\Lambda_2$ starts at $(1,z_1)$ and ends at $(z_2,n)$;
	\item the third sequence, $\Lambda_3,$ is obtained by matching the remaining $\ctd(n)-\ctd(z_0)$ $D$s in $X$ with the first $\ctd(n)-\ctd(z_0)$ $R$s in $Y$ to create an increasing sequence starting from $(\pd(\ctd(z_0)+1),\pr(1))$ and ending at $(n,z_3)$ (note that thanks to Petrov conditions, for $n$ big enough, $\ctd(n)-\ctd(z_0)$ is smaller than the total number of $R$s in $Y$ and so the operations is well-defined);
	\item the fourth sequence, $\Lambda_4,$ is obtained by matching the remaining $\ctu(n)-(\ctl(n)-\ctd(z_0))$ $U$s in $X$ with the remaining $\ctr(n)-(\ctd(n)-\ctd(z_0))$ $R$s in $Y$ (note that $\ctu(n)-(\ctl(n)-\ctd(z_0))=\ctr(n)-(\ctd(n)-\ctd(z_0))$ since $\ctu(n)+\ctd(n)=n$ and $\ctl(n)+\ctr(n)=n$) to create a decreasing sequence between $(z_2,n)$ and $(n,z_3)$ (this two boundary points are not included in $\Lambda_4$ by definition).
\end{itemize} 
\medskip
We can conclude that, for $n$ big enough, the index corresponding to each $D$ and $U$ in $X$ and each $L$ and $R$ in $Y$ are used exactly once. Therefore by this construction each index of $X$ is paired with a unique index in $Y$ so the resulting collection of points must correspond to the points of a permutation $\sigma:[n]\to[n].$ 
We will show (see Lemma \ref{lem:map_welldef} below) that in fact $\sigma$ is in $Sq(n)$ and so that the assignment $\rho((X,Y,z_0))\coloneqq\sigma$ define a map from $\Omega_n$ to $Sq(n)$ when $n$ is big enough.

\begin{figure}\includegraphics[scale=.5]{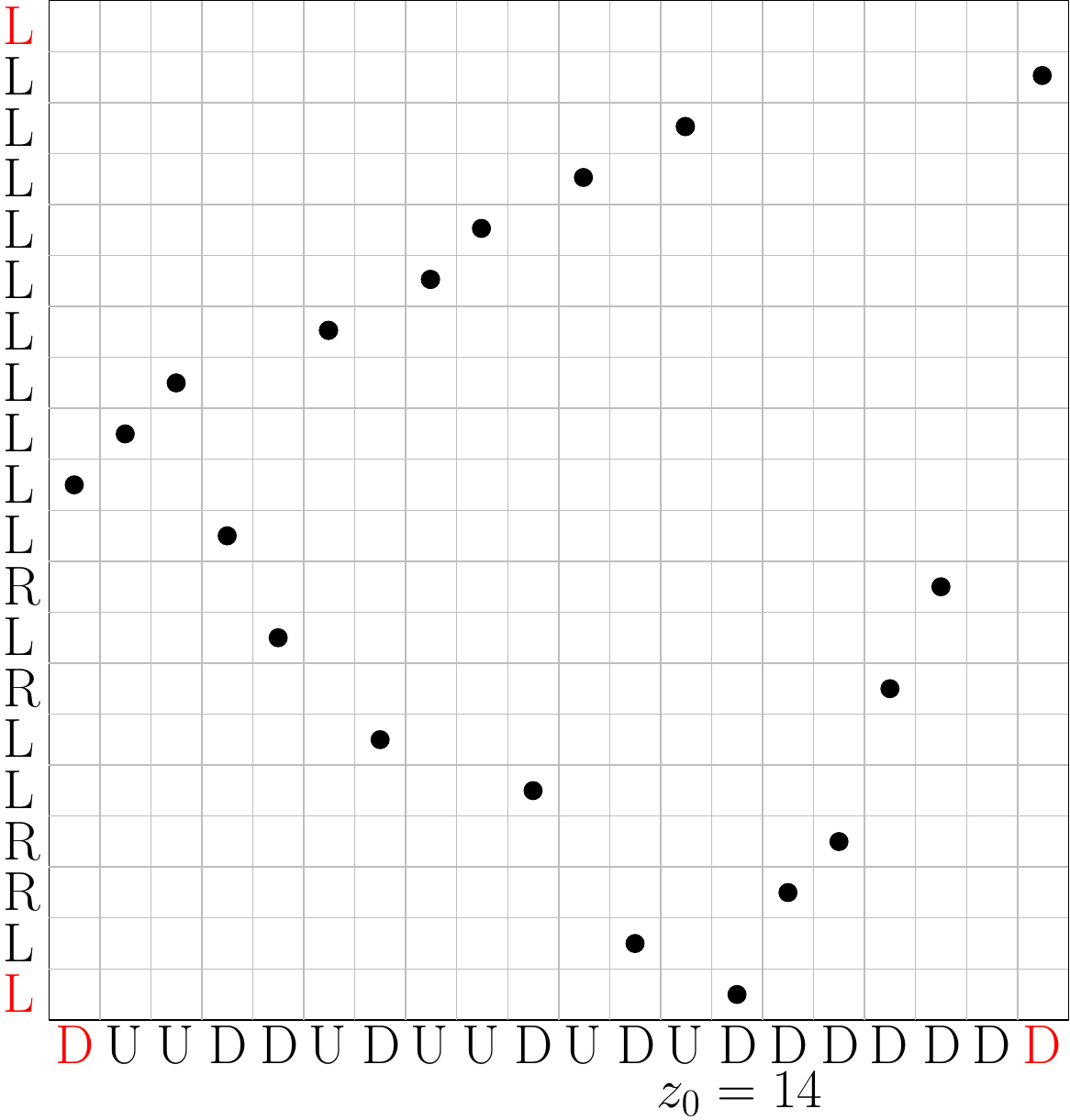}
	\caption{An example where the construction of $\sigma$ from a good anchored pair of sequences fails.  The last $L$ does not have a corresponding $U$ with which to match. The problem is that the sequences $X$ (resp.\ $Y$) contains too many $D$s (resp.\ too many $L$s) and so it does not satisfy the Petrov conditions.}
	\label{fail1}	
\end{figure}

Note that for some good anchored sequences $(X,Y,z_0)$ that are not in $\Omega_n$ it is possible that the construction of the permutation $\sigma$ might fail (see Fig.~\ref{fail1}).

The Petrov conditions force even stricter conditions on the sequences $\Lambda_i,$ for $i=1,2,3,4$.  
\begin{lemma}\label{guiding light}
	If $(X,Y,z_0)\in \Omega_n$, then 
	\begin{itemize}
	\item $|s+t -z_0| < 10n^{.6}$ for $(s,t) \in \Lambda_1$,	
	\item $ |t-s - z_0| < 10n^{.6}$ for $(s,t)  \in \Lambda_2$,
 	\item $ |s-t - z_0| < 10n^{.6}$ for $(s,t)  \in \Lambda_3$,
	\item $ |2n- s-t-z_0| < 10n^{.6}$ for $(s,t)  \in \Lambda_4$.
	\end{itemize}
\end{lemma}

\begin{proof}
We give details for points in $\Lambda_1$.  The rest of the proof follows through similar arguments.  Let $i\in I_1$ so that $(\pd(i), \pl(\ctd(z_0) - i + 1)) \in \Lambda_1.$  By using the Petrov conditions twice (once with item (6) for $\pd$ and $\pl$ and once with items (5) for $\ctd$) we obtain
\begin{equation*}
|\pd(i) + \pl(\ctd(z_0)-i+1) - z_0| < |2i + 2\ctd(z_0) -2i + 2 - z_0| + 4n^{.6}< |z_0 +2 -z_0| + 5n^{.6}< 10n^{.6}.	
\end{equation*}
The same argument applies for points in $\Lambda_2$, $\Lambda_3$, and $\Lambda_4$.
\end{proof}

\begin{lemma}
	\label{lem:map_welldef}
For $n$ big enough the following holds. Let  $(X,Y,z_0) \in \Omega_n$ and  let $\sigma=\rho(X,Y,z_0)$ as above.  Then $\sigma$ is in $Sq(n)$ with $\LRm = \Lambda_1$, $\LRM = \Lambda_2 \cup \{(1,z_1)\}$, $\RLm = \Lambda_3 \cup \{(z_0,1)\}$ and $\RLM = \Lambda_4\cup \{(z_2,n),(n,z_3)\}.$  Moreover, $\rho$ is injective with $\phi\circ\rho$ acting as the identity on $\Omega_n$.   
\end{lemma}

\begin{proof}
  Note that $\sigma$ is in $Sq(n)$ if the sets $\Lambda_1$, $\Lambda_2 \cup \{(1,z_1)\}$, $\Lambda_3\cup\{(z_0,1)\}$ and $\Lambda_4\cup\{(z_2,n),(n,z_3)\}$ correspond to the record sets $\LRm$, $\LRM$, $\RLm$, and $\RLM$ respectively.  We will focus only on showing that $\LRm = \Lambda_1$.  The proofs of the remaining correspondences will follow in a similar manner.  

\begin{figure}
\includegraphics[scale=.5]{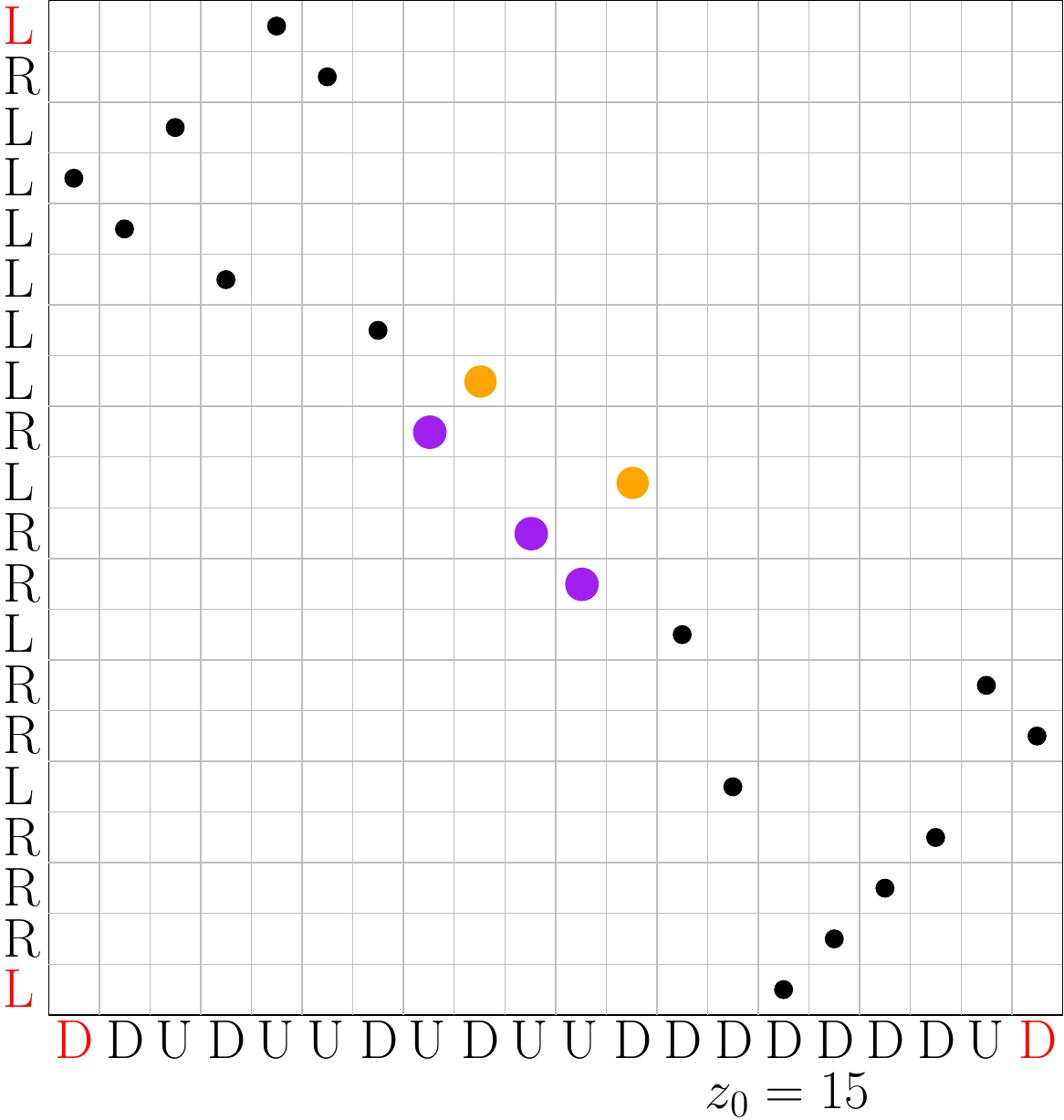}	
\caption{An example of a construction of $\sigma$ where $\Lambda_1$ and $\LRm$ do not agree.  The orange bigger points are in $\Lambda_1$ but not in $\LRm$.  The purple bigger points are in $\Lambda_4$ but not $\RLM.$  } 
\label{fail2}
\end{figure}

Suppose $(i,\sigma(i))\in \Lambda_1$ is not in $\LRm$.  Then there exists $j<i$ such that $\sigma(j) < \sigma(i).$  The points in $\Lambda_1$ are decreasing so $(j,\sigma(j))$ cannot be in $\Lambda_1$.  If $(j,\sigma(j)) \in \Lambda_2$ then $\sigma(j) > \sigma(i)$.  If $(j,\sigma(j)) \in \Lambda_3$ then $j>z_0>i.$  Thus this may only happen if $(j,\sigma(j)) \in \Lambda_4$ (see Fig.~\ref{fail2}).
 
By Lemma \ref{guiding light}
\begin{equation}\label{thing1}
\sigma(i) + i < z_0 + 10n^{.6}.
\end{equation}
Similarly for $(j,\sigma(j)) \in \Lambda_4$, 
\begin{equation}\label{thing2}
\sigma(j) +j > 2n - z_0 - 10n^{.6}.
\end{equation}
Subtracting \eqref{thing2} from \eqref{thing1} gives for $n$ big enough
\begin{equation}\label{consequences}
	\sigma(i) - \sigma(j) + i-j < 2z_0 - 2n + 20n^{.6} < -2n^{.9} +20n^{.6} < 0.
\end{equation}
Note that if $i>j$ and $\sigma(i) >\sigma(j)$ then the left hand side of \eqref{consequences} would be positive.  This contradiction shows that $\Lambda_1 \subseteq \LRm$.  The same argument shows that $\Lambda_4\subseteq \RLM$.  Similar arguments show that $\Lambda_2\subseteq \LRM$ and $\Lambda_3\subseteq\RLm.$  

Finally we note that if $(i,\sigma(i)) \in \LRm$ but not $\Lambda_1$, then it would have to be in $\Lambda_4$ as $\Lambda_1$ contains the corners $(1,z_1)$ and $(z_0,1)$.  This would imply $(i,\sigma(i)) \in \RLM\cap \LRm$ and thus satisfy $i+ \sigma(i) = n+1$.  Plugging this into \eqref{thing2} creates a contradiction.  Therefore $ \LRm = \Lambda_1$ and similar equalities hold for the other three sets if we ignore the appropriate corners.

Lastly we note that under the map $\phi$, $z_0 = \sigma^{-1}(1)$, the corners $(z_0,1),(1,z_1),(z_2,n)$, and $(n,z_3)$ project to the appropriate labels $(D,L)$, $(D,L)$ $(U,L)$ and $(D,R)$, non-corner points of $\Lambda_1$, $\Lambda_2$, $\Lambda_3$ and $\Lambda_4$ project onto $(D,L)$, $(U,L)$, $(D,R)$, and $(U,R)$ respectively, so $\phi(\sigma)$ agrees with $(X,Y, z_0)$.  Thus we may conclude that $\phi\circ\rho$ is the identity on $\Omega_n$ which also implies that $\rho$ is injective from $\Omega_n$ into $Sq(n)$.
\end{proof}
 
We conclude this section with the following key result.
 
\begin{lemma}\label{square_is_rect}
With probability $1-o(1)$ a uniform random square permutation $\bm{\sigma}_n$ of size $n$ is in $\rho(\Omega_n)$.	
\end{lemma}

\begin{proof}

The map $\rho$ is injective from $\Omega_n$ into $Sq(n)$ and thus $\P(\bm{\sigma}_n \in \rho(\Omega_n)) = \frac{|\Omega_n|}{|Sq(n)|}$.

First we note that the negative term in \eqref{sqenum} satisfies $(2n-5){ 2n-6 \choose n-3 } = o(2(n+2)4^{n-3})$, so the size of $Sq(n)$ satisfies
\begin{equation}
\label{eq:est_sqn}
|Sq(n)|=2(n+2)4^{n-3}(1-o(1)).
\end{equation}
By (\ref{karl}) the number of good anchored pairs of sequences of size $n$ is $2(n+2)4^{n-3}.$ Therefore, using Lemma \ref{omega_size} the size of $\Omega_n$ satisfies 
$$|\Omega_n|=2(n+2)4^{n-3}(1-o(1)).$$

Thus, as $n$ tends to infinity, $\P(\bm{\sigma}_n\in \rho(\Omega_n)) \to 1.$  
\end{proof}

\begin{remark}
	\label{noenumneeded} Note that we used the enumeration of the set of square permutations stated in (\ref{sqenum}) to obtain the estimate in (\ref{eq:est_sqn}), but actually, in order to prove the above lemma it was enough to know that $\rho$ is injective and that $|Sq(n)|\leq 2(n+2)4^{n-3}$. The latter is a consequence of the fact that the map $\phi$ defined in Section $\ref{sect:perm_to_anchored_seq}$ is injective (this was proved in \cite{duchi_square2}). So, the techniques of \cite{duchi_square2} and of the current paper allow to derive the first order term of the enumeration of $Sq(n)$, which is enough to prove Lemma $\ref{square_is_rect}$.
	As a consequence, this suggests that for other classes where the exact enumeration is not known, a similar approach may yield interesting asymptotic enumerative results. These can then be used to establish some other probabilistic results. For instance, this approach has been used in \cite{borga2019almost}.
\end{remark}

\section{Global behavior}
\label{sect:glob_beha}

In this section we consider the global behavior of a random square permutation by studying its permuton limit.  For an exhaustive introduction to the permuton convergence we refer to \cite[Section 2]{bassino2017universal}.

A \emph{permuton} $\mu$ is a Borel probability measure on the unit square $[0,1]^2$ with uniform marginals, that is
\[
\mu( [0,1] \times [a,b] ) = \mu( [a,b] \times [0,1] ) = b-a,
\]
for all $0 \le a \le b\le 1$. Any permutation $\sigma$ of size $n \ge 1$ may be interpreted as a permuton $\mu_\sigma$ given by the sum of area measures
\[
\mu_\sigma= n \sum_{i=1}^n \Leb\left([(i-1)/n, i/n]\times[(\sigma(i)-1)/n,\sigma(i)/n]\right).
\]

Let $\mathcal M$ be the set of permutons.
We need to equip $\mathcal M$ with a topology.
We recall that a sequence of (deterministic) permutons $(\mu_n)_n$ converges \emph{weakly} to $\mu$ (simply denoted $\mu_n \to \mu$) if 
$$
\int_{[0,1]^2} f d\mu_n \to \int_{[0,1]^2} f d\mu,
$$
for every bounded and continuous function $f: [0,1]^2 \to \mathbb{R}$. With this topology, $\mathcal M$ is compact and metrizable by the metric $d_{\square}$ defined, for every pair of permutons $(\mu,\mu'),$ by
$$d_{\square}(\mu,\mu')=\sup_{R\in\mathcal{R}}|\mu(R)-\mu'(R)|,$$
where $\mathcal R$ denotes the set of rectangles contained in $[0,1]^2.$ Once we have a topology for deterministic permutons we can define the convergence for random permutations as follows.

\begin{definition}
	We say that a random permutation $\bm{\sigma}_n$ converges in distribution to a random permuton $\bm{\mu}$ as $n \to \infty$ if the random permuton $\mu_{\bm{\sigma}_n}$ converges in distribution to $\bm{\mu}$ with respect to the topology defined above.  
\end{definition}

There are a few main steps in establishing convergence in distribution in the permuton topology for uniform elements of $Sq(n).$  Lemma \ref{square_is_rect} shows that it suffices to consider only permutations $\sigma_n$ in $\rho(\Omega_n)$.  Then we show in Lemma \ref{permuton_bounds} a uniform bound for the distance between the permuton $\mu_{\sigma_n}$ and a certain rectangular permuton with bottom corner at $(\sigma_{n}^{-1}(1)/n,0)$.  Finally we show the main result in Theorem \ref{thm:perm_conv}.   
 
\medskip 
 
First we define our candidate limiting permuton. Let $z$ be a point in $[0,1]$.  Let $L_1$ and $L_4$ denote the line segments with slope $-1$ connecting $(0,z)$ to $(z,0)$ and $(1-z,1)$ to $(1,1-z)$, respectively.  Similarly let $L_2$ and $L_3$ denote the line segments with slope $1$ connecting $(0,z)$ to $(1-z,1)$ and $(z,0)$ to $(1,1-z)$, respectively.  The union of $L_1$, $L_2$, $L_3$ and $L_4$ forms a rectangle in $[0,1]^2.$  

For each of the line segments $L_i$ ($i=1,2,3$, or $4$) we will define a measure $\mu^z_i$ as a rescaled Lebesgue measure. Let $\nu$ be the Lebesgue measure on $[0,1]$.  Let $S$ be a Borel measurable set on $[0,1]^2$.  For each $i$, let $S_i = S\cap L_i$.  Finally let $\pi_x(S_i)$ be the projection of $S_i$ onto the $x$-axis and $\pi_y(S_i)$ the projection onto the $y$-axis.  As each line has slope $1$ or $-1$, the measures of the projections satisfy $\nu(\pi_x(S_i)) = \nu(\pi_y(S_i)).$  For each $i=1,2,3,4$, define $\mu^z_i(S) := \frac{1}{2} \nu( \pi_x( S_i ) ) = \frac12 \nu( \pi_y(S_i)).$

Finally we define the measure $\mu^z = \mu^z_1+\mu^z_2+\mu^z_3+\mu^z_4.$

\begin{lemma} \label{permutonian}
	The measure $\mu^z$ is a permuton.
\end{lemma}       

\begin{proof}

By construction $\mu^z$ is a measure.  Then all that is left is to check that for $0\leq a< b\leq 1$,  $\mu^z([0,1]\times [a,b] )=\mu^z([a,b]\times [0,1]) = b-a.$  Let $\mu^z_{up} = \mu^z_2 + \mu^z_4$, $\mu^z_{down} = \mu^z_1+\mu^z_3$, $\mu^z_{left} = \mu^z_1+\mu^z_2$ and finally $\mu^z_{right} = \mu^z_3 + \mu^z_4.$  

Let $S= [a,b]\times[0,1]$.  The projection $\pi_x(S_1)$ is either $[a,b]$, $[a,z]$ or $\emptyset$ depending on whether $z>b$, $a\leq z \leq b$, or $z<a$ respectively.  Similarly $\pi_x(S_3)$ is either $\emptyset$, $[z,b]$, or $[a,b]$ if $z>b$, $a\leq z\leq b$, or $z< a$ respectively.  Thus for any choice of $z$, $\mu^z_{down} = \frac12 (b-a).$  Similarly, $\mu^z_{up}(S) = \frac12 (b-a)$ and $\mu^z(S) = \mu^z_{up}(S) + \mu^z_{down}(S) = b-a$ as desired.  The same argument holds for the projection $\pi_y$ on the set $\bar{S} = [0,1]\times [a,b]$ with respect to $\mu^z_{right}$ and $\mu^z_{left}$, to show $\mu(\bar{S}) = \mu^z_{right}(\bar{S}) + \mu^z_{left}(\bar{S}) = b-a$, finishing the proof.     	
\end{proof}

The following lemma shows that for $\sigma_n\in \rho(\Omega_n)$ the permutons, $\mu_{\sigma_n}$ and $\mu^{z_n}$ with $z_n = \sigma_n^{-1}(1)/n,$ have distance $d_{\square}(\mu_{\sigma_n},\mu^{z_n})$ that tends to zero as $n$ tends to infinity, uniformly over all choices of $\sigma_n$.

\begin{lemma}\label{permuton_bounds}
Let $\sigma_n \in \rho(\Omega_n)$ and let $z_n =\sigma_n^{-1}(1)/n$. Then for $n$ big enough
	$$\sup_{\sigma_n \in \Omega_n} d_{\square}(\mu_{\sigma_n},\mu^{z_n}) < 400n^{-.4}.$$
\end{lemma}

\begin{proof}

Fix $\sigma_n \in \Omega_n$ and $R= (a,b)\times(c,d) \subset [0,1]^2$.  The permutation $\sigma_n$ will consist of four disjoint sets of points $\Lambda_i$ for $i=1,2,3$, and $4$ (see the discussion before Lemma \ref{guiding light}).  For each of these sets of points we define the measure $\lambda_i$ on $[0,1]^2$ as 
$$\lambda_i := \frac{1}{n}\sum_{(i,j) \in \Lambda_i} \Leb\left([(i-1)/n,i/n)\times[(j-1)/n,j/n)\right).$$
%\quad\text{for all Borel sets } R\subseteq [0,1]^2 .$$

Noting that $\mu_{\sigma_n}=\lambda_1+\lambda_2+\lambda_3+\lambda_4$ we have the bound $|\mu_{\sigma_n}(R) - \mu^{z_n}(R)| \leq \sum_{i=1}^4 |\lambda_i(R)- \mu^{z_n}_i(R)|.$
We will show explicitly that $|\lambda_1(R) - \mu^{z_n}_1(R)| < 100n^{-.4}$.  Similar arguments show the same bound for $i=2,3,4.$    

Recall $L_1$ is the line connecting $(0,z_n)$ and $(z_n,0)$.  Let $\ell$ denote the line segment given by $R\cap L_1$ (that we assume non-empty) with end points at $(x_1,y_1)$ and $(x_2,y_2)$ where $x_1\leq  x_2$ and $y_1 \geq y_2.$  These endpoints satisfy $x_1+y_1 = x_2 + y_2 = z_n$.  By this construction we have $\mu^{z_n}_1(R) = \frac{1}{2}(x_2-x_1).$ 

Let $(s_1,t_1)$ be the leftmost point in $\Lambda_1 \cap nR$ and $(s_2,t_2)$ the rightmost point.  The total number of points in $\Lambda_1\cap nR$ is given by $\ctd(s_2) - \ctd(s_1) +1$. Therefore $\lambda_1(R)=\tfrac{1}{n}(\ctd(s_2) - \ctd(s_1) +1)\pm\frac{\varepsilon}{n},$ with $\varepsilon\leq2$ (the error term comes form the first and last area measures) and so $$|\lambda_1(R)-\mu^{z_n}_1(R)|<\left|\frac{1}{n}(\ctd(s_1)-1) - \frac{1}{2} x_1 \right|+\left|\frac{1}{n}\ctd(s_2) - \frac{1}{2} x_2 \right|+\frac{\varepsilon}{n}.$$

 By Lemma \ref{guiding light}, the points of $\Lambda_1$ must lie between the lines $nL_1^-$ and $nL_1^+$ given by the equations  $x+y = nz_n \pm 20n^{.6}$ (\emph{cf.}\ Fig.~\ref{schema_for_proof}).

Suppose $L_1$ exits $(a,b)\times(c,d)$ from the top so that $y_1 = d$.  Then points in $\Lambda_1$ with first coordinate in the interval $[nx_1-40n^{.6}, nx_1-20n^{.6}]$ must lie above the line $y = nd$.  Similarly points in $\Lambda_1$ in the interval $[nx_1+20n^{.6},nx_1 + 40n^{.6}]$ must lie below the line $y=nd$.  By Lemma \ref{petrov_often} there is at least one point in $\Lambda_1$ with $x$-coordinate in each of these intervals.  Thus the leftmost point $(s_1,t_1)$ must have $s_1$ in the interval $[nx_1 - 40n^{.6}, nx_1 + 40n^{.6}].$  This combined with the Petrov conditions shows that 
$$\left|\ctd(s_1) - \frac{1}{2}nx_1\right| < \left|\ctd(s_1) - \frac{1}{2}s_1\right| + \left|\frac{1}{2}(s_1 -  nx_1)\right|< n^{.6} + 40n^{.6}$$ and thus $$\left |\frac{1}{n}(\ctd(s_1)-1) - \frac{1}{2} x_1 \right | < 50n^{-.4}.$$  
\begin{figure}
	\includegraphics[scale=.8]{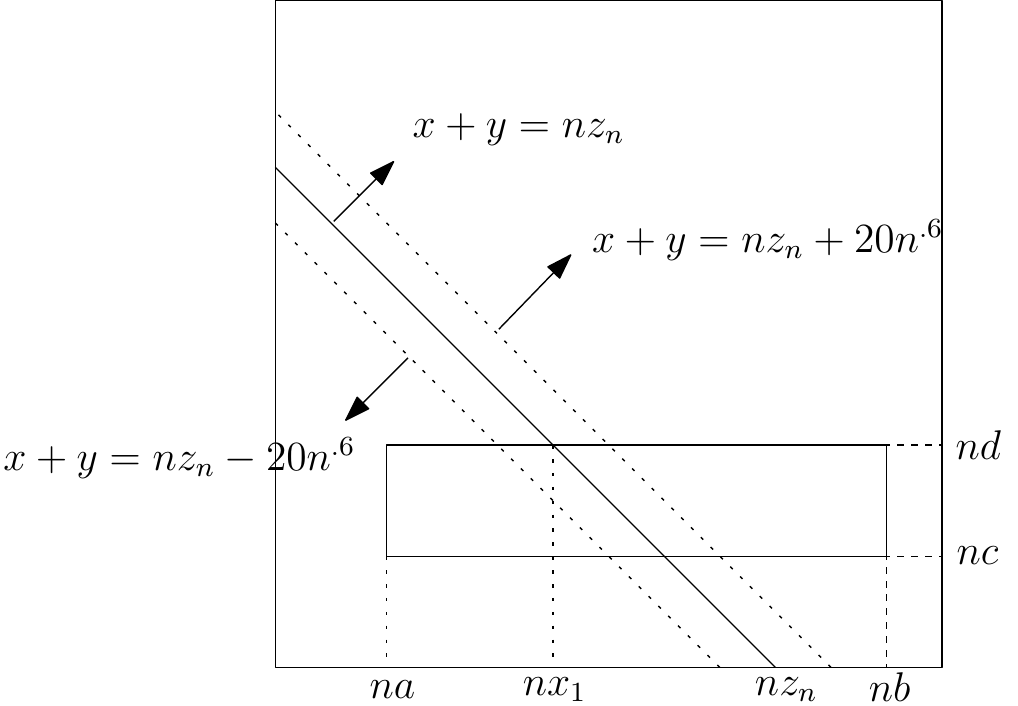}
	\caption{A diagram for the proof of Lemma \ref{permuton_bounds}.} 
	\label{schema_for_proof}
\end{figure}
If $L_1$ exits $nR$ on the left so that $x_1 = a$, then a similar argument leads to the same conclusion.  Likewise we can show that $$\left |\frac{1}{n}(\ctd(s_2)) - \frac{1}{2} x_2 \right | < 50n^{-.4}$$ and thus 
$$\left |\lambda_1 (R) -  \mu_1^{z_n}(R)\right| < 100n^{-.4}.$$ 

Similarly, for each $i=2,3,4$, $|\lambda_i(R) -  \mu^{z_n}_i(R)| < 100n^{-.4}$ and thus
\begin{equation} \label{specific_bound}
	\left |\mu_{\sigma_n}(R) - \mu^{z_n}(R)\right| < 400n^{-.4}.
\end{equation} 
This bound is uniform over all $\sigma_n \in \Omega_n$ and $R\in [0,1]^2$ and so concludes the proof.  
\end{proof}

For $\zz$ uniformly random on $(0,1)$, we have a corresponding random permuton $\mu^{\zz}$.  This is precisely our permuton limit for $\bm{\sigma}_n \in Sq(n)$.

\begin{theorem}
\label{thm:perm_conv}
Let $\bm{\sigma}_n$ be a uniform random element of $Sq(n)$ and let $\zz$ be a uniform element in $(0,1)$.  The random permuton $\mu_{\bm{\sigma}_n}$ converges in distribution to the random permuton $\mu^{\zz}.$	
\end{theorem}

\begin{proof}
By Lemma \ref{square_is_rect} it suffices to only consider permutation chosen uniformly from $\rho(\Omega_n)$ when showing the distributional limit of $\mu_{\bm{\sigma}_n}.$  

Let ${\zz}_n = \bm{\sigma}_n^{-1}(1)/n$.  Since $\bm{\sigma}_n$ is uniform in $\rho(\Omega_n)$ then $\zz_n$ is uniform in $(n^{-.1},1-n^{-.1})$ and so converges in distribution to $\zz$. The map $z\to\mu^z$ is continuous as a function form $(0,1)$ to $\mathcal{M}$, and thus $\mu^{\zz_n}$ converges in distribution to $\mu^{\zz}.$
By Lemma \ref{permuton_bounds}, we also have that $d_{\square}(\mu_{\bm{\sigma}_n},\mu^{\zz_n})$ converges almost surely to zero. Therefore, combining these results,  we can conclude that $\mu_{\bm{\sigma}_n}$ converges in distribution to $\mu^{\zz}.$
\end{proof}

\section{Fluctuations}
\label{sect:fluctuations}
We saw in Section~\ref{sect:glob_beha} that the permuton limit of a sequence of uniform random square permutations is a random rectangle. We now want to study the fluctuations of the dots of the diagram of a uniform square permutation around the four edges of the rectangle.

\subsection{Statement of the main result and outline of the proof}

Let $\sigma_n$ be a square permutation of size $n$ and let $z_0=z_0(n)=\sigma_n^{-1}(1).$  We assume that 
\begin{equation}
z_0>\frac{n}{2}+10n^{.6}.
\label{eq:key_assumption}
\end{equation}

We will focus on the following three families of points of $\sigma_n$:
\begin{itemize}
	\item $DR=DR(n)=\RLm(\sigma_n),$ highlighted in green in Fig.~\ref{Squarepermsimulation};
	\item $DL=DL(n)=\{(i,\sigma_n(i))\in\LRm(\sigma_n):\sigma_n(i)\leq n-z_0+1\},$ highlighted in red in Fig.~\ref{Squarepermsimulation};
	\item $UR=UR(n)=\{(i,\sigma_n(i))\in\RLM(\sigma_n):i\geq z_0\},$ highlighted in blue in Fig.~\ref{Squarepermsimulation}.
\end{itemize}

Note that $(z_0,1)\in DR\cap DL$ and $(n,\sigma_n(n))\in DR\cap UR.$

For each set of points, we perform a particular rotation so that each of the following lines become the new $x$-axis for the respectively set of points
\begin{itemize}
	\item $r^{DR}: y=x+(1-z_0),$ highlighted in green in Fig.~\ref{Squarepermsimulation};
	\item $r^{DL}: y=-x+(z_0+1),$ highlighted in red in Fig.~\ref{Squarepermsimulation};
	\item $r^{UR}: y=-x+(2n-z_0+1),$ highlighted in blue in Fig.~\ref{Squarepermsimulation}.
\end{itemize} 

More precisely, as shown in Fig.~\ref{Squarepermsimulation}, we apply a clockwise rotation of 45 degrees to the first family of points, a clockwise rotation of 135 degrees to the second family, and counter-clockwise rotation of 45 degrees to the third family. Note that the first two sequences of points (obtained form $DR$ and $DL$) starts at height zero. In order to have the same for the third sequence, we translate the $y$-coordinate of all the points in the third family by the distance of the first point from the line $r^{UR}$. We denote the three new families of points as $\mathcal{P}^{DR}$, $\mathcal{P}^{DL}$ and $\mathcal{P}^{UR}.$

\begin{figure}[htbp]
	\begin{center}
		\includegraphics[scale=1.5]{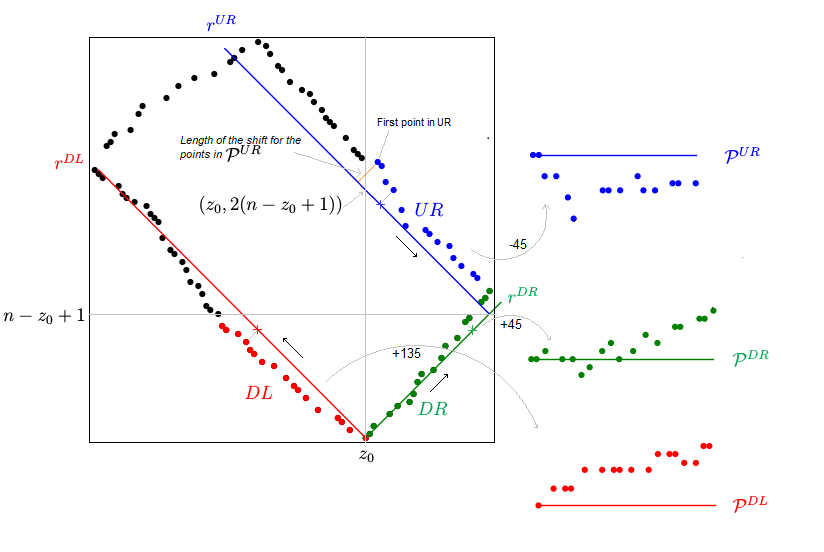}
		\caption{A square permutation $\sigma$ with the three families $DR,DL,UR$ highlighted. The dots of $DR,DL,UR$ are colored in the diagram of $\sigma$ in green, red and blue respectively. Similarly we paint the lines $r^{DR}$, $r^{DL}$ and $r^{UR}$ in green, red and blue. On the right of the picture we draw the diagrams of the points in $\mathcal{P}^{UR}$, $\mathcal{P}^{DR}$ and $\mathcal{P}^{DL}$ obtained rotating the families of points $UR$, $DR$ and $DL$ by the indicated angle (with the additional translation for the points in the family $\mathcal{P}^{UR}$).}
		\label{Squarepermsimulation}
	\end{center}
\end{figure}

Given a family of points $\mathcal{P}=\{(x_i,y_i)\}_{i=0}^m$, with $x_0\leq x_1\leq x_2\leq\dots\leq x_m$, we denote with $F^{\mathcal{P}}(t),$ for $t\in[0,1],$ the linear interpolation among the points $\{(\frac{x_i}{x_m},\frac{y_i}{\sqrt{m}})\}_{i=0}^m.$

Let $\mathcal{C}([0,1],\mathbb{R})$ denotes the space of continuous functions from the interval $[0,1]$ to $\mathbb{R},$ endowed with the uniform distance. Recall also that $z_0=\sigma_n^{-1}(1).$

\begin{theorem}
	\label{thm:fluctuations}
	Let $\bm{\sigma}_n$ be a uniform random square permutation of size $n$, and let $\bm{B}_1(t),\bm{B}_2(t),\bm{B}_3(t)$, and $\bm{B}_4(t)$ be four independent standard Brownian motions on the interval $[0,1].$ Fix a sequence of integers $(t_n)_n$ such that $\frac{n}{2}+10n^{.6}<t_n\leq n-n^{.9}.$ Conditioning on $\bm{z_0}=t_n,$ we have the following convergence in distribution in the space $\mathcal{C}([0,1],\mathbb{R})$:
	\begin{equation*}
	\big(\bm{F}^{\mathcal{P}^{DR(n)}}(t),\bm{F}^{\mathcal{P}^{DL(n)}}(t),\bm{F}^{\mathcal{P}^{UR(n)}}(t)\big)_{t\in[0,1]}\stackrel{d}{\longrightarrow}\big(\bm{B}_1(t)+\bm{B}_2(t),\bm{B}_3(t)+\bm{B}_1(t),\bm{B}_4(t)+\bm{B}_2(t)\big)_{t\in[0,1]}.
	\end{equation*} 
\end{theorem}

\begin{remark}
	Note that Theorem $\ref{thm:fluctuations}$ not only describes the scaling limit of the families of points $\mathcal{P}^{DR}$, $\mathcal{P}^{DL}$ and $\mathcal{P}^{UR}$, but also describes the dependency relations among them. Indeed, the limit $\lim_{n\to\infty}\bm{F}^{\mathcal{P}^{DL(n)}}(t)=\bm{B}_3(t)+\bm{B}_1(t)$ is independent of the limit $\lim_{n\to\infty}\bm{F}^{\mathcal{P}^{UR(n)}}(t)=\bm{B}_4(t)+\bm{B}_2(t).$ On the other hand, the limit $\lim_{n\to\infty}\bm{F}^{\mathcal{P}^{DR(n)}}(t)=\bm{B}_1(t)+\bm{B}_2(t)$ is correlated with both the limits $\lim_{n\to\infty}\bm{F}^{\mathcal{P}^{DL(n)}}(t)=\bm{B}_3(t)+\bm{B}_1(t)$ and $\lim_{n\to\infty}\bm{F}^{\mathcal{P}^{UR(n)}}(t)=\bm{B}_4(t)+\bm{B}_2(t)$. Moreover, this dependency is completely explicit.
\end{remark}

\begin{remark}
	We chose to study only the family of points $DL,$ $DR$ and $UR$ in order to simplify as much as we can the notation. Nevertheless, the result stated in Theorem \ref{thm:fluctuations} can be generalized to every possible choice of "a vertical and horizontal strip" in the diagram (under the assumption that they do not contain corner points). In particular, in our case, the vertical strip is the one between the indexes $z_0$ and $n$ and the horizontal strip is the one between the values $1$ and $n-z_0+1$.
\end{remark}

In order to prove Theorem \ref{thm:fluctuations} we will consider a uniform random permutation $\bm{\sigma}_n$ of $\rho(\Omega_n)$ with $\bm{\sigma}_n^{-1}(1)>\frac{n}{2}+10n^{.6}$. We now compute the $x$-coordinates and the $y$-coordinates of the points in the three families $DR,$ $DL$ and $UR$ for a permutation $\sigma=\rho(X,Y,z_0)$ with $(X,Y,z_0)\in\Omega_n$ and $z_0>\frac{n}{2}+10n^{.6}.$ Specifically, we are going to write the $x$-coordinates and the $y$-coordinates of the points in terms of the sequences $X$ and $Y$. 

\begin{observation}
	Using Lemma \ref{okz} the condition $z_0>\frac{n}{2}+10n^{.6}$ ensures that $z_2\coloneqq\sigma^{-1}(n)<\sigma^{-1}(1)\eqqcolon z_0,$ and therefore every $U$ that appears after the $z_0$-th position in the sequence $X$ is used to create only right-to-left maxima in $\sigma$.
\end{observation}

We start by noting that $|DR|=\ctd(n)-\ctd(z_0)+1$, $|DL|=\ctl(n-z_0+1)$ and $|UR|=\ctu(n)-\ctu(z_0)+1.$ We denote the points in $DR$ (resp.\ $DL,UR$) with the letters $\{P^{DR}_i\}^{|DR|-1}_{i=0}$ (resp.\ $\{P^{DL}_i\}^{|DL|-1}_{i=0},$ $\{P^{UR}_i\}^{|UR|}_{i=1}$) in such a way that $P^{DR}_0=(z_0,0)=P^{DL}_0$ and $P^{UR}_1$ is the point in $UR$ with smallest $x$-coordinate\footnote{Note that the indices of the points in the families $|DR|$ and $|DL|$ start from zero, but the indices of the points in $|UR|$ start from one. These choices are made in order to have some simplifications in the following computations.}. We are indexing the points respecting the orders indicated by the three small black arrows in Fig.~\ref{Squarepermsimulation}. 

For all $i\geq 1,$ we have
\begin{itemize}
	\item $P^{DR}_i=(z_0+\pdzr(i),\pr(i)),$ where $\pdzr(i):=\pd(\ctd(z_0)+i)-z_0$;
	\item $P^{DL}_i=(z_0-\pdzl(i),\pl(i+1)),$ where $\pdzl(i):=z_0-\pd(\ctd(z_0)+i)$;
	\item $P^{UR}_i=(z_0+\puzr(i),\pr(n-z_0+1-i)),$ where $\puzr(i):=\pu(\ctu(z_0)+i)-z_0.$
\end{itemize}
Note that $\pdzr(i)$ denotes the distance from $z_0$ of the $i$-th $D$ after the one at position $z_0.$ Similar remarks hold for $\pdzl(i)$ and $\puzr(i)$.

With some easy computations, we can rewrite the $x$-coordinates and the $y$-coordinates of the points in $\mathcal{P}^{DR}$, $\mathcal{P}^{DL}$ and $\mathcal{P}^{UR}$ as follow
\begin{itemize}
	\item $P^{\mathcal{P}^{DR}}_i=\frac{\sqrt{2}}{2}\big(\pdzr(i)+\pr(i)-1,-\pdzr(i)+\pr(i)-1\big)$;
	\item $P^{\mathcal{P}^{DL}}_i=\frac{\sqrt{2}}{2}\big(\pdzl(i)+\pl(i+1)-1,\pdzl(i)-\pl(i+1)+1\big)$; 
	\item  $P^{\mathcal{P}^{UR}}_i=\frac{\sqrt{2}}{2}\big(x_i,y_i\big),$ where
	$x_i=\puzr(i)+2n-2z_0+1-\pr(n-z_0+1-i)$ and \break $y_i=\puzr(i)-\puzr(1)+\pr(n-z_0+1-i)-\pr(n-z_0).$
\end{itemize} 
Note that the $y$-coordinates of all the points in the three families depend both from elements coming from the sequence $X$ and the sequence $Y$. Therefore, for each family, we split this dependence introducing the following six additional families of points (see also Fig.~\ref{Schema_notation})
\begin{itemize} 
	\item Set $P_i^{\mathcal{X}^{DR}}=\big(i,-\pdzr(i)+2i\big)$ and $\mathcal{X}^{DR}=\big\{P_i^{\mathcal{X}^{DR}}\big\}_{i=0}^{|DR|-1}$;
	\item set $P_i^{\mathcal{Y}^{DR}}=\big(i,\pr(i)-1-2i\big)$ and  $\mathcal{Y}^{DR}=\big\{P_i^{\mathcal{Y}^{DR}}\big\}_{i=0}^{|DR|-1}$;
	\item set $P_i^{\mathcal{X}^{DL}}=\big(i,\pdzl(i)-2i\big)$ and  $\mathcal{X}^{DL}=\big\{P_i^{\mathcal{X}^{DL}}\big\}_{i=0}^{|DL|-1}$;
	\item set $P_i^{\mathcal{Y}^{DL}}=\big(i,-\pl(i+1)+1+2i\big)$ and  $\mathcal{Y}^{DL}=\big\{P_i^{\mathcal{Y}^{DL}}\big\}_{i=0}^{|DL|-1}$;
	\item set $P_i^{\mathcal{X}^{UR}}=\big(i,\puzr(i)-\puzr(1)-2i\big)$ and 
	$\mathcal{X}^{UR}=\big\{P_i^{\mathcal{X}^{UR}}\big\}_{i=1}^{|UR|}$;
	\item set $P_i^{\mathcal{Y}^{UR}}=\big(i,\pr(n-z_0+1-i)-\pr(n-z_0)+2i\big)$ and 
	$\mathcal{Y}^{UR}=\big\{P_i^{\mathcal{Y}^{UR}}\big\}_{i=1}^{|UR|}$.
\end{itemize} 

Note that the $y$-coordinates of the points in $\mathcal{P}^{*}$, for $*=DR,DL,UR$, are respectively the sum of the $y$-coordinates of the points in $\mathcal{X}^{*}$ and $\mathcal{Y}^{*}$ (up to the factor $\tfrac{\sqrt{2}}{2}$).

We will prove Theorem \ref{thm:fluctuations} as follows
\begin{itemize}
	\item The six families $\mathcal{X}^{*}$ and $\mathcal{Y}^{*}$, for $*=DR,DL,UR$ have to be thought as a sort of projection of the points in the families $\mathcal{P}^{*}$ on the $x$-axis and the $y$-axis of the diagram (see Fig.~\ref{Schema_notation}).
	\item We will first study (see Proposition \ref{prop:prop_for_thm} below) the scaling limits of the six families $\mathcal{X}^{DR}$, $\mathcal{Y}^{DR}$, $\mathcal{X}^{DL}$, $\mathcal{Y}^{DL}$, $\mathcal{X}^{UR}$ and $\mathcal{Y}^{UR}$, proving the convergence of the functions $\bm{F}^{\mathcal{X}^{*}}(t)$ and $\bm{F}^{\mathcal{Y}^{*}}(t)$ to six standard Brownian motions (multiplied by a factor $\sqrt{2}$). In particular the functions $\bm{F}^{\mathcal{Y}^{DL}}(t)$ and $\bm{F}^{\mathcal{Y}^{DR}}(t)$ converge to the same Brownian motion $\bm{B}_1$ and the families $\bm{F}^{\mathcal{X}^{DR}}(t)$ and $\bm{F}^{\mathcal{X}^{UR}}(t)$ converge to the same Brownian motion $\bm{B}_2$. 
	\item Then we recover the scaling limit for the families $\mathcal{P}^{DR}$, $\mathcal{P}^{DL}$ and $\mathcal{P}^{UR}$ as a linear combinations of the limits obtained for the previous six families. In this case we have to overcome a technical difficulty due to the fact that the points in the families $\mathcal{P}^{DR}$, $\mathcal{P}^{DL}$ and $\mathcal{P}^{UR}$ are random in both coordinates. This problem is solved in Lemma \ref{lem:technical_prob_lemma}.
\end{itemize}

\begin{figure}[htbp]
	\begin{center}
		\includegraphics[scale=0.6]{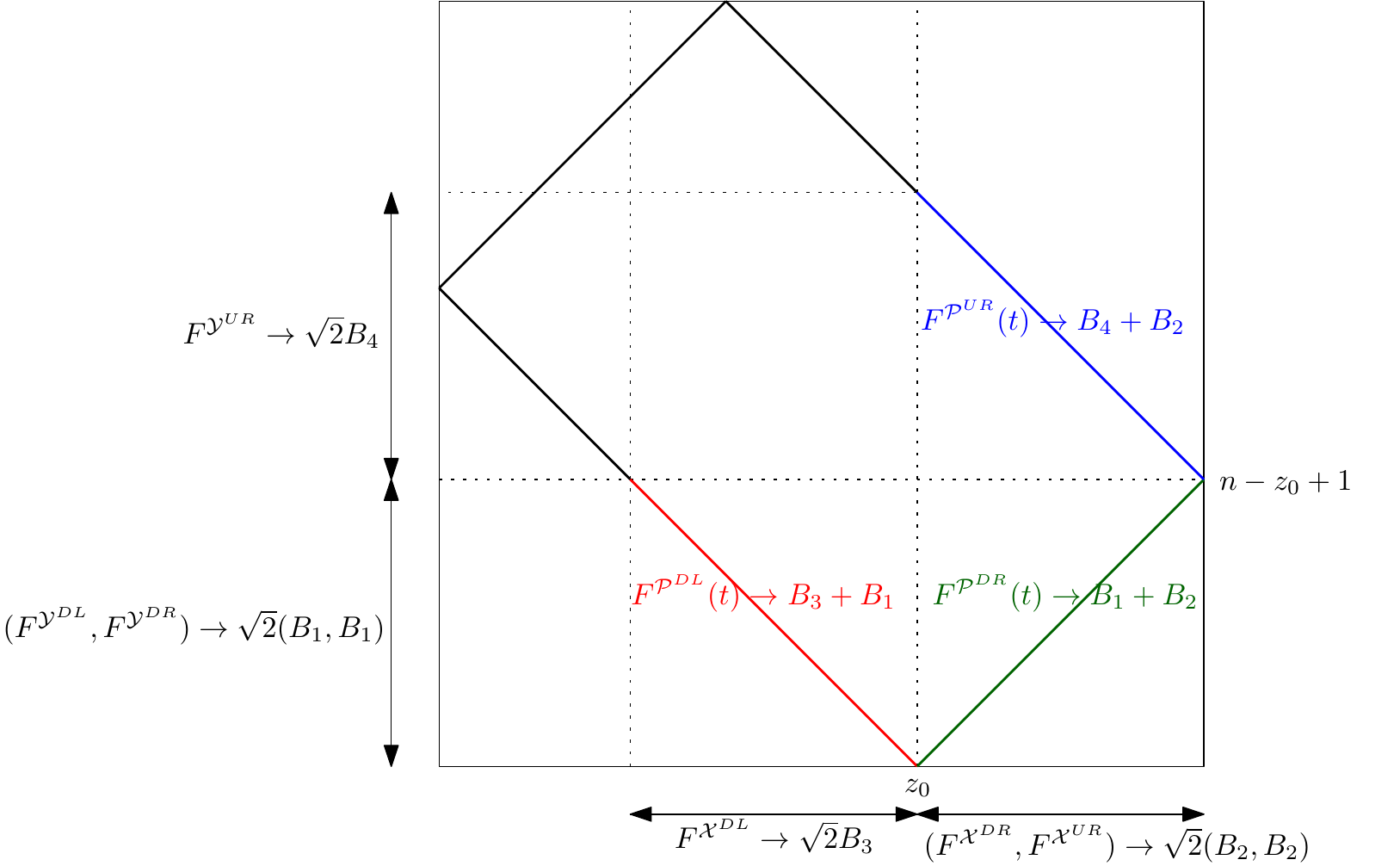}
		\caption{A schema summarizing the strategy of the proof for Theorem \ref{thm:fluctuations}.}
		\label{Schema_notation}
	\end{center}
\end{figure}

\subsection{Preparation lemmas} The goal of this section is to prove the following result.

\begin{proposition}
	\label{prop:prop_for_thm}
	Let $\bm{\sigma}_n$, $(t_n)_n$ and $\bm{B}_1(t),\bm{B}_2(t),\bm{B}_3(t),\bm{B}_4(t)$ be defined as in Theorem~\ref{thm:fluctuations}. Conditioning on $\bm{z_0}=t_n,$ we have the following convergences in distribution in the space $\mathcal{C}([0,1],\mathbb{R})$: 
	\begin{equation}
	\label{eq:Goal_of_prop_for_thm}
	\begin{split}
	\big(\bm{F}^{\mathcal{Y}^{UR(n)}}(t),\bm{F}^{\mathcal{Y}^{DL(n)}}(t),\bm{F}^{\mathcal{Y}^{DR(n)}}(t),&\bm{F}^{\mathcal{X}^{DL(n)}}(t),\bm{F}^{\mathcal{X}^{DR(n)}}(t),\bm{F}^{\mathcal{X}^{UR(n)}}(t)\big)_{t\in[0,1]}\\
	&\stackrel{d}{\longrightarrow}\sqrt{2}\big(\bm{B}_4(t),\bm{B}_1(t),\bm{B}_1(t),\bm{B}_3(t),\bm{B}_2(t),\bm{B}_2(t)\big)_{t\in[0,1]}.\\
	\end{split}
	\end{equation} 
\end{proposition}

Before proving the proposition we have to solve a technical difficulty due to the fact that our families of points have random cardinalities.

\begin{lemma}
	\label{lem:tecnical_lemma}
	Let $(\bm X_i)_{i\in\N}$ be a sequence of i.i.d random variables with zero mean and unit variance. Let $\bm S_m=\sum_{i=1}^m\bm X_i.$ For all $m\in\N,$ consider an integer-valued random variable $\bm N=\bm N(m)$ such that a.s. $|\bm N-m|<2m^{.6}$. We also set $\bm{\mathcal{P}}=\{(i,\bm S_i)\}_{i=0}^{m+\lfloor 2m^{.6}\rfloor}$ and $\bm{\mathcal{P}}'=\{(i,\bm S_i)\}_{i=0}^{\bm N}$. Then, as $m\to\infty$,
	$$\sup_{t\in[0,1]}|F^{\bm{\mathcal{P}}}(t)-F^{\bm{\mathcal{P}}'}(t)|\stackrel{a.s.}{\longrightarrow}0$$
\end{lemma}

\begin{proof} We first show that 
	$$\sup_{t\in[0,1]}|F^{\bm{\mathcal{P}}}(t)-F^{\bm{\mathcal{P}}'}(t)|\leq \sup_{\substack{i\in[0,m+\lfloor 2m^{.6}\rfloor]\\ j\in[-\bm \delta,\bm \delta]\\
	i+j\in[0,m+\lfloor 2m^{.6}\rfloor]}}\Big|\tfrac{\bm S_i}{\sqrt{m+\lfloor 2m^{.6}\rfloor}}-\tfrac{\bm S_{i+j}}{\sqrt{\bm N}}\Big|,$$
	where $\bm \delta=m+\lfloor 2m^{.6}\rfloor-\bm N+2$. For that, it is enough to note that for every fixed $t\in[0,1],$
	$$|F^{\bm{\mathcal{P}}}(t)-F^{\bm{\mathcal{P}'}}(t)\big|\leq\max_{s,q\in\{0,1\}}\left|\frac{\bm S_{\lfloor t(m+\lfloor 2m^{.6}\rfloor)\rfloor+s}}{\sqrt{m+\lfloor 2m^{.6}\rfloor}}-\frac{\bm S_{\lfloor t\bm N\rfloor+q}}{\sqrt{\bm N}}\right|,$$
	and that $\max_{s,q\in\{0,1\}}\{\lfloor t(m+\lfloor 2m^{.6}\rfloor)\rfloor+s-\lfloor t\bm N\rfloor+q\}\leq m+\lfloor 2m^{.6}\rfloor-\bm N+2=\bm \delta.$
	
	Therefore it is enough to show that
		$$\sup_{\substack{i\in[0,m+\lfloor 2m^{.6}\rfloor]\\ j\in[-\bm \delta,\bm \delta]\\ i+j\in[0,m+\lfloor 2m^{.6}\rfloor]}}\Big|\tfrac{\bm S_i}{\sqrt{m+\lfloor 2m^{.6}\rfloor}}-\tfrac{\bm S_{i+j}}{\sqrt{\bm N}}\Big|\stackrel{a.s.}{\longrightarrow}0.$$

	Note that 
	\begin{equation}
	\label{eq:starting_point}
	\Big|\tfrac{\bm S_i}{\sqrt{m+\lfloor 2m^{.6}\rfloor}}-\tfrac{\bm S_{i+j}}{\sqrt{\bm N}}\Big|\leq\frac{|\bm S_i-\bm S_{i+j}|}{\sqrt{\bm N}}+\Big|\tfrac{\sqrt{\bm N}}{\sqrt{m+\lfloor 2m^{.6}\rfloor}}-1\Big|\frac{|\bm S_i|}{\sqrt{\bm N}}.
	\end{equation}
	By the law of iterated logarithms the random variable
	$$\bm M\coloneqq\sup_{n\geq 2}\frac{\bm{S}_n}{\sqrt{n\log\log n}},$$
	is almost surely finite. Therefore, almost surely, for all $i\in[0,m+\lfloor 2m^{.6}\rfloor],$
	$$\sup_{j\in[-\bm \delta,\bm \delta]}|\bm S_i-\bm S_{i+j}|\leq\bm M\sqrt{\bm \delta\log\log \bm\delta}.$$
	Since $\frac{\bm M\sqrt{\bm \delta\log\log \bm\delta}}{\sqrt{\bm N}}\stackrel{a.s.}{\longrightarrow}0$ we can conclude that
	$$\sup_{\substack{i\in[0,m+\lfloor 2m^{.6}\rfloor]\\ j\in[-\bm \delta,\bm \delta]\\ i+j\in[0,m+\lfloor 2m^{.6}\rfloor]}}\frac{|\bm S_i-\bm S_{i+j}|}{\sqrt{\bm N}}\stackrel{a.s.}{\longrightarrow}0.$$
	Similarly,
	\begin{multline*}
	\Big|\tfrac{\sqrt{\bm N}}{\sqrt{m+\lfloor 2m^{.6}\rfloor}}-1\Big|\sup_{\substack{i\in[0,m+\lfloor 2m^{.6}\rfloor]}}\frac{|\bm S_i|}{\sqrt{\bm N}}\\
	\leq\Big|\tfrac{\sqrt{\bm N}}{\sqrt{m+\lfloor 2m^{.6}\rfloor}}-1\Big|\bm{M}\frac{\sqrt{(m+\lfloor 2m^{.6}\rfloor)\log\log(m+\lfloor 2m^{.6}\rfloor)}}{\sqrt{\bm N}}\stackrel{a.s.}{\longrightarrow}0.
	\end{multline*}
	The last two equations together with the initial bound in  (\ref{eq:starting_point}) are enough to conclude the proof.
\end{proof}

\begin{proof}[Proof of Proposition \ref{prop:prop_for_thm}]
	It is enough to prove the statement for a uniform random element $\bm{\sigma}_n$ of $\rho(\Omega_n)$. Then the statement for uniform square permutations follows using Lemma \ref{square_is_rect}.
	
	We recall that a uniform random element of $\rho(\Omega_n)$ can be sampled as follows.
	Let $\zz_0$ be an integer chosen uniformly from $(n^{.9},n-n^{.9})$ and let $(\bmX,\bmY)$ be uniform in $\{U,D\}^n\times \{L,R\}^n$ conditioned to satisfy the Petrov conditions and to have $\bmY_{\zz_0}=D$. Under these assumptions $(\bmX,\bmY,\bm z_0)$ is a uniform random element of $\Omega_n$ and consequently $\bm{\sigma}_n = \rho((\bmX,\bmY,\bm z_0))$ is a uniform random element of $\rho(\Omega_n)$. All the random quantities considered below have to be meant as conditioned to $\bm z_0=t_n$. 
	
	We start by proving that
	\begin{equation}
	\label{eq:goal_of_the_proof}
	\big(\bm{F}^{\mathcal{Y}^{DL(n)}}(t),\bm{F}^{\mathcal{Y}^{DR(n)}}(t)\big)_{t\in[0,1]}\stackrel{d}{\longrightarrow}\sqrt{2}\big(\bm{B}_1(t),\bm{B}_1(t)\big)_{t\in[0,1]}.
	\end{equation}
	We recall that that $\bm{F}^{\mathcal{Y}^{DL(n)}}(t)$ and $\bm{F}^{\mathcal{Y}^{DR(n)}}(t)$ are the functions obtained by linear interpolation of the points in the families
	\begin{equation}
	\label{eq:interpolating_points}
	\begin{split}
	&\Big\{\Big(\frac{i}{|\bm{DL}(n)|},\frac{-\bm{\pl}(i+1)+1+2i}{\sqrt{|\bm{DL}(n)|}}\Big)\Big\}_{i=0}^{|\bm{DL}(n)|-1},\\
	&\Big\{\Big(\frac{i}{|\bm{DR}(n)|},\frac{\bm{\pr}(i)-1-2i}{\sqrt{|\bm{DR}(n)|}}\Big)\Big\}_{i=0}^{|\bm{DR}(n)|-1},
	\end{split}
	\end{equation}
	where $|\bm{DL}(n)|=\bm{\ctl}(n-\bm z_0+1)$ and $|\bm{DR}(n)|=\bm{\ctd}(n)-\bm{\ctd}(\bm z_0)+1$.
	
	In order to prove (\ref{eq:goal_of_the_proof}) we are going to apply Donsker's theorem. Therefore it is enough to prove that the differences between the $y$-coordinates of two consecutive points are independent and identically distributed. We also have to pay a bit of attention to the fact that the families of points have random cardinalities.
	
	Using similar notation as the one introduced immediately before Lemma \ref{lemm: technical_lemma_for_fluctuation} (for a sequence of $L$s and $R$s instead of $U$s and $D$s), we can rewrite the numerator of the $y$-coordinates in the two families in  (\ref{eq:interpolating_points}) as
	\begin{equation}
	\label{eq:pointdistrib}
	\begin{split}
	&\bm{\pr}(i)-1-2i=\bm{e}(i)-1,\quad\forall i\leq \bm{\ctr}(n),\\ &-\bm{\pl}(i+1)+1+2i=\bm{e}(i+1)-\bm{s}(i+1)-1, \quad\forall i\leq\min\{\bm{\ctr}(n)-1,\bm{\ctl}(n)-1\}.
	\end{split}
	\end{equation}
	Using Lemma \ref{lemm: technical_lemma_for_fluctuation} we have that for all $i\leq\min\{\bm{\ctr}(n)-1,\bm{\ctl}(n)-1\},$
	\begin{equation}
	\label{eq:s_bound}
	|\bm{s}(i)|<10n^{.4}\quad\text{a.s.}
	\end{equation} 
	
	We also note that, for all $i<\bm{\ctr}(n)$, the random variable $\bm{e}(i+1)-\bm{e}(i)$ associated with the random sequence $\bmY$ has the following distribution, independent of $i$,
	\begin{equation}
	\label{eq:distribution}
	\P(\bm{e}(i+1)-\bm{e}(i)=k)=\P(\bm{\pr}(i+1)-\bm{\pr}(i)=k+2)=\Big(\frac{1}{2}\Big)^{k+2},\quad\text{for all}\quad k \geq -1.
	\end{equation}
	In the last equality (thanks to Lemma \ref{omega_size}) we used that the random sequence $\bmY$ is a uniform sequence in $\{L,R\}^n,$ although the sequence $\bmY$ is conditioned to satisfy the Petrov conditions.
	In particular $(\bm{e}(i+1)-\bm{e}(i))_{i}$ are i.i.d.\ random variables with zero mean and variance equal to 2.
	
By the Petrov conditions a.s.
	\begin{equation}
	\begin{split}
	\label{eq:cardinalities}
	\big||\bm{DR}(n)|-(n-\bm{z_0})/2\big|<|n-\bm{z_0}|^{.6}\quad&\text{and}\quad\big||\bm{DL}(n)|-(n-\bm{z_0})/2\big|<|n-\bm{z_0}|^{.6},\\
	|\bm{\ctr}(n)-n/2|<n^{.6}\quad&\text{and}\quad|\bm{\ctl}(n)-n/2|<n^{.6}.
	\end{split}
	\end{equation}	
Since $\bm{z_0}=t_n>\frac{n}{2}+10n^{.6},$ we a.s. have that for $n$ big enough
	\begin{equation*}
	\max\{|\bm{DR}(n)|,|\bm{DL}(n)|\}\leq\min\{\bm{\ctr}(n),\bm{\ctl}(n)\}-\frac{n}{10},
	\end{equation*}
	and therefore the relations in (\ref{eq:pointdistrib}), (\ref{eq:s_bound}) and (\ref{eq:distribution}) hold (for $n$ big enough) a.s.\ for all the points in the sets in (\ref{eq:interpolating_points}).
	
Inequality (\ref{eq:s_bound}) guarantees that if $\bm{F}^{\mathcal{Y}^{DR(n)}}(t)$ is obtained by interpolating the points in the family $\Big\{\Big(\frac{i}{|\bm{DR}(n)|},\frac{\bm{e}(i+1)-1}{\sqrt{|\bm{DR}(n)|}}\Big)\Big\}_{i=0}^{|\bm{DR}(n)|-1}$ instead of the original points $\Big\{\Big(\frac{i}{|\bm{DR}(n)|},\frac{\bm{e}(i+1)-\bm s(i+1)-1}{\sqrt{|\bm{DR}(n)|}}\Big)\Big\}_{i=0}^{|\bm{DR}(n)|-1}$, then the distributional limit is the same.
	
	We now consider for all $m\geq1$ two additional functions $\bm{F}^{m}_1(t)$ and $\bm{F}^{m}_2(t)$, obtained by linear interpolation of the points in the families
	\begin{equation}
	\label{eq:interpolating_points2}
	\begin{split}
	&\Big\{\Big(\frac{i}{m},\frac{\bm{e}(i)-1}{\sqrt{m}}\Big)\Big\}_{i=0}^{m},\\
	&\Big\{\Big(\frac{i}{m},\frac{\bm{e}(i+1)-1}{\sqrt{m}}\Big)\Big\}_{i=0}^{m}.
	\end{split}
	\end{equation}
	Applying Donsker's theorem, we have that
	\begin{equation}
	\label{eq:conv_to_bm}
	\big(\bm{F}^{m}_1(t),\bm{F}^{m}_2(t)\big)\stackrel{d}{\longrightarrow}\sqrt{2}\big(\bm{B}_1(t),\bm{B}_1(t)\big).
	\end{equation}

	Using the inequalities in (\ref{eq:cardinalities}) and applying Lemma \ref{lem:tecnical_lemma} with $\bm N=|\bm{DL}(n)|$ (resp.\ $\bm N=|\bm{DR}(n)|$) and $m=\lfloor(n-\bm{z_0})/2\rfloor$, we have that 
	$$\sup_{t\in[0,1]}\left|\bm{F}^{\mathcal{Y}^{DL(n)}}(t)-\bm{F}^{\lfloor(n-\bm{z_0})/2+|n-\bm{z_0}|^{.6}\rfloor}_1(t)\right|\stackrel{a.s.}{\longrightarrow}0,$$ 
	$$\sup_{t\in[0,1]}\left|\bm{F}^{\mathcal{Y}^{DR(n)}}(t)-\bm{F}^{\lfloor(n-\bm{z_0})/2+|n-\bm{z_0}|^{.6}\rfloor}_2(t)\right|\stackrel{a.s.}{\longrightarrow}0.$$
	This with (\ref{eq:conv_to_bm}) implies  (\ref{eq:goal_of_the_proof}).
	
	Repeating the same proof as before for $\bm{F}^{\mathcal{Y}^{UR(n)}}(t)$, $\bm{F}^{\mathcal{X}^{DL(n)}}(t)$ and $\big(\bm{F}^{\mathcal{X}^{DR(n)}}(t),\bm{F}^{\mathcal{X}^{UR(n)}}(t)\big)$, we can conclude that
	\begin{equation*}
	\begin{split}
	\bm{F}^{\mathcal{Y}^{UR(n)}}(t)\stackrel{d}{\longrightarrow}\sqrt{2}\bm{B}_4(t),&\quad \bm{F}^{\mathcal{X}^{DL(n)}}(t)\stackrel{d}{\longrightarrow}\sqrt{2}\bm{B}_3(t),\\ \big(\bm{F}^{\mathcal{X}^{DR(n)}}(t),\bm{F}^{\mathcal{X}^{UR(n)}}(t)&\big)\stackrel{d}{\longrightarrow}\sqrt{2}(\bm{B}_2(t),\bm{B}_2(t)).
	\end{split}
	\end{equation*}
	Finally, noting that the following four families of points are asymptotically independent
	\begin{equation*}
	\bm{\mathcal{Y}^{UR}},\quad \bm{\mathcal{Y}^{DL}}\cup\bm{\mathcal{Y}^{DR}},\quad \bm{\mathcal{X}^{DL}},\quad \bm{\mathcal{X}^{DR}}\cup\bm{\mathcal{X}^{UR}},
	\end{equation*}
	we can immediately deduce that the following four functions are asymptotically independent
	\begin{equation*}
	\bm{F}^{\mathcal{Y}^{UR(n)}}(t),\quad \big(\bm{F}^{\mathcal{Y}^{DL(n)}}(t),\bm{F}^{\mathcal{Y}^{DR(n)}}(t)\big),\quad \bm{F}^{\mathcal{X}^{DL(n)}}(t),\quad \big(\bm{F}^{\mathcal{X}^{DR(n)}}(t),\bm{F}^{\mathcal{X}^{UR(n)}}(t)\big),
	\end{equation*}
	and so we can conclude that the joint convergence in distribution in  (\ref{eq:Goal_of_prop_for_thm}) holds.
\end{proof}

\subsection{The proof of the main result}
Before proving Theorem \ref{thm:fluctuations} we need to state an additional technical lemma. We first introduce some more notation.

Given a family of points $\mathcal{P}=\{(x_i,y_i)\}_{i=0}^m$, with $x_1\leq x_2\leq\dots\leq x_m$, we denote with $F_{Y}^{\mathcal{P}}(t),$ for $t\in[0,1],$ the linear interpolation among the points $\{(\frac{i}{m},\frac{y_i}{\sqrt{m}})\}_{i=0}^m$ and with $F_{X}^{\mathcal{P}}(t),$ for $t\in[0,1],$ the linear interpolation among the points $\{(\frac{x_i}{x_m},\frac{i}{m})\}_{i=0}^m.$

Note that 
\begin{equation}
\label{eq:trivial_obs_comp}
F^{\mathcal{P}}(t)=F_{Y}^{\mathcal{P}}\circ F_{X}^{\mathcal{P}}(t).
\end{equation}

\begin{lemma}
	\label{lem:technical_prob_lemma}
	Let for all $n\in\N,$ $\bm{\mathcal{P}}_n=\{(\bm{x}_i,\bm{y}_i)\}_{i=0}^{\bm{N}(n)}$ be a family of $\bm{N}(n)$ random points (where $\bm{N}(n)$ is a positive integer-valued random variable). Assume that the following two convergences hold in the space $\mathcal{C}([0,1],\mathbb{R}),$ 
	$$\bm{F}_{X}^{\bm{\mathcal{P}}_n}(t)\stackrel{d}{\longrightarrow}  Id_{[0,1]}(t)\quad\text{and}\quad \bm{F}_{Y}^{\bm{\mathcal{P}}_n}(t)\stackrel{d}{\longrightarrow} \bm{B}(t),$$ 
	where $Id_{[0,1]}(t)$ is the identity function on the interval $[0,1]$ and $\bm{B}(t)$ is a standard Brownian motion. Then 
	$$\bm{F}^{\bm{\mathcal{P}}_n}(t)\stackrel{d}{\longrightarrow} \bm{B}(t).$$	
\end{lemma}

\begin{proof}
	Using Skorokhod's representation theorem we can assume that
	$$\bm{F}_{X}^{\bm{\mathcal{P}}_n}(t)\stackrel{a.s.}{\longrightarrow}  Id_{[0,1]}(t)\quad\text{and}\quad \bm{F}_{Y}^{\bm{\mathcal{P}}_n}(t)\stackrel{a.s.}{\longrightarrow} \bm{B}(t),$$ 
	in the space $\mathcal{C}([0,1],\mathbb{R})$ equipped with the uniform distance.
	 
	We recall that if $f_n$ and $g_n$ are two sequences of functions in the space $\mathcal{C}([0,1],\mathbb{R})$ that uniformly converge to $f$ and $g$ respectively, then the sequence $f_n\circ g_n$ uniformly converges to $f \circ g.$
	Therefore, using the observation done in  (\ref{eq:trivial_obs_comp}), we can conclude that
	\[\bm{F}^{\bm{\mathcal{P}}_n}(t)=\bm{F}_Y^{\bm{\mathcal{P}}_n}(t)\circ\bm{F}_X^{\bm{\mathcal{P}}_n}(t)\stackrel{a.s.}{\longrightarrow} \bm{B}(t).\qedhere\]	
\end{proof}

We also need the following easy lemma.
\begin{lemma}
	\label{lem:bound_pertov}
For a regular anchored pair of sequences with associated set of points $DR(n)$, and for all $i\leq |DR(n)|,$
	\[|\pdzr(i)+\pr(i)-1-4i|<4n^{.6}+1.\] 
\end{lemma}

\begin{proof}
	Recalling that $\pdzr(i)=\pd(\ctd(z_0)+i)-z_0$ and that $\pd(\ctd(z_0))=z_0$ (since $X_{z_0}=D$) we have
	\[|\pdzr(i)-2i|=|\pd(\ctd(z_0)+i)-z_0-2i|=|\pd(\ctd(z_0)+i)-\pd(\ctd(z_0))-2i|<2n^{.6},\] 
	where in the last inequality we used the Petrov conditions.
	Using again the Petrov conditions we also have
	\[|\pr(i)-2i|<2n^{.6}.\]
	The two bounds are enough to conclude the proof.
\end{proof}

We can finally prove the main result of this section.

\begin{proof}[Proof of Theorem \ref{thm:fluctuations}]
	
	We first prove that
	\begin{equation}
	\label{eq:main_goal_of_the_proof}
	\bm{F}^{\mathcal{P}^{DR(n)}}(t)\stackrel{d}{\longrightarrow}\bm{B}_1(t)+\bm{B}_2(t).
	\end{equation}
	We recall that that $\bm{F}^{\mathcal{P}^{DR(n)}}(t)$ is the function obtained by linear interpolation of the family of points
	\begin{equation*}
	\Big\{\frac{\sqrt{2}}{2}\Big(\frac{\bm \pdzr(i)+\bm \pr(i)-1}{\bm \pdzr(|\bm{DR}(n)|)+\bm \pr(|\bm{DR}(n)|)-1},\frac{-\bm \pdzr(i)+\bm \pr(i)-1}{\sqrt{|\bm{DR}(n)|}}\Big)\Big\}_{i=0}^{|\bm{DR}(n)|-1},
	\end{equation*}
	where $|\bm{DR}(n)|=\bm{\ctd}(n)-\bm{\ctd}(z_0)+1.$

	Note that the points are random in both coordinates. In order to overcome this difficulty we are going to apply Lemma \ref{lem:technical_prob_lemma}. Therefore, in order to prove  (\ref{eq:main_goal_of_the_proof}) it is enough to prove that
	\begin{equation}
	\label{eq:first_goal_of_proof}
	\bm{F}_X^{\mathcal{P}^{DR(n)}}(t)\stackrel{d}{\longrightarrow}Id_{[0,1]}(t),
	\end{equation} 
	and 
	\begin{equation}
	\label{eq:second_goal_of_proof}
	\bm{F}_Y^{\mathcal{P}^{DR(n)}}(t)\stackrel{d}{\longrightarrow}\bm{B}_1(t)+\bm{B}_2(t).
	\end{equation} 
	
	The convergence in  (\ref{eq:second_goal_of_proof}) follows from Proposition \ref{prop:prop_for_thm}. Indeed note that 
	\[ \bm{F}_Y^{\mathcal{P}^{DR(n)}}(t)=\frac{\sqrt{2}}{2}\big(\bm{F}^{\mathcal{X}^{DR(n)}}(t)+\bm{F}^{\mathcal{Y}^{DR(n)}}(t)\big)\stackrel{d}{\longrightarrow}\bm{B}_1(t)+\bm{B}_2(t).\]
	
	For the convergence in  (\ref{eq:first_goal_of_proof}) note that by Lemma \ref{lem:bound_pertov}, for all $i\leq |\bm{DR}(n)|,$
	\[|\bm{\pdzr}(i)+\bm{\pr}(i)-1-4i|<4n^{.6}+1\quad a.s.\] 
	Therefore,
	\[ \sup_{i\leq |\bm{DR}(n)|}\Big|\frac{\bm \pdzr(i)+\bm \pr(i)-1}{\bm \pdzr(|\bm{DR}(n)|)+\bm \pr(|\bm{DR}(n)|)-1}-\frac{i}{|\bm{DR}(n)|}\Big|\stackrel{a.s.}{\longrightarrow}0.\]
	This implies that
	\[ \sup_{t\in[0,1]}\Big|\bm{F}_X^{\mathcal{P}^{DR(n)}}(t)-Id_{[0,1]}(t)\Big|\stackrel{a.s.}{\longrightarrow}0,\]
	and so  (\ref{eq:first_goal_of_proof}) is proved.
	Using the joint converge proved in Proposition \ref{prop:prop_for_thm} and the same proof that we used to prove  (\ref{eq:main_goal_of_the_proof}) we have that \begin{equation}
	\big(\bm{F}^{\mathcal{P}^{DR(n)}}(t),\bm{F}^{\mathcal{P}^{DL(n)}}(t),\bm{F}^{\mathcal{P}^{UR(n)}}(t)\big)_{t\in[0,1]}\stackrel{d}{\longrightarrow}\big(\bm{B}_1(t)+\bm{B}_2(t),\bm{B}_3(t)+\bm{B}_1(t),\bm{B}_4(t)+\bm{B}_2(t)\big)_{t\in[0,1]},
	\end{equation} 
	concluding the proof.
\end{proof}

\section{Local behavior}
\label{sect:local_beh}
For local behavior of uniform random permutations in $Sq(n)$ we use the setting of local topology for permutations introduced in \cite[Section 2]{borga2018local}.
We now briefly recall the definition of this topology. 

\subsection{Local topology for permutations}

A \emph{finite rooted permutation} is a pair $(\sigma,i),$ where $\sigma\in\SG^n$ and $i\in[n]$ for some $n\in\N.$ We denote  with $\SG^n_{\bullet}$ the set of rooted permutations of size $n$ and with $\SG_{\bullet}:=\bigcup_{n\in\N}\SG^n_{\bullet}$ the set of finite rooted permutations. We write sequences of finite rooted permutations in $\SG_{\bullet}$ as $(\sigma_n,i_n)_{n\in\N}.$ 

To a rooted permutation $(\sigma,i),$ we associate (as shown in the right-hand side of Fig.~\ref{fig:rest}) the pair $(\Asi,\leqsi),$  where $\Asi:=[-i+1,|\sigma|-i]$ is a finite interval containing 0 and $\leqsi$ is a total order on $\Asi,$ defined for all $\ell,j\in \Asi$ by
\begin{equation*}
\ell\leqsi j\qquad\text{if and only if}\qquad \sigma({\ell+i})\leq\sigma({j+i})\;.
\end{equation*} 
Informally, the elements of $\Asi$ should be thought as the column indices of the diagram of $\sigma$,
shifted so that the root is in column $0$.
The order $\leqsi$ then corresponds to the vertical order on the dots in the corresponding columns.
\begin{figure}[htbp]
	\begin{center}
		\includegraphics[scale=.70]{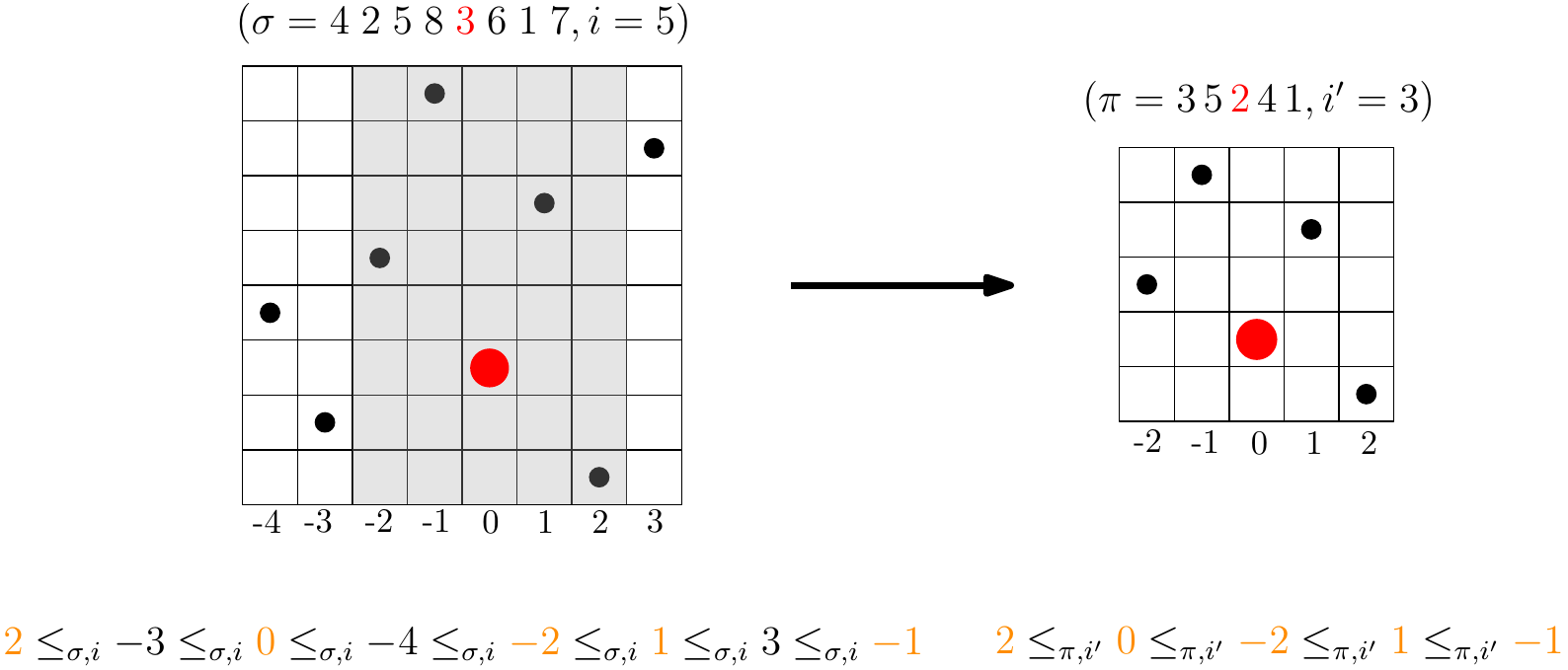}\\
		\caption{Two rooted permutations and the associated total orders. The big red dot indicates the root of the permutation. The vertical gray strip and the relation between the two rooted permutations will be clarified later.}
		\label{fig:rest}
	\end{center}
\end{figure}

This map is a bijection from the space of finite rooted permutations $\SG_{\bullet}$ to the space of total orders on finite integer intervals containing zero. Consequently and throughout this Section \ref{sect:local_beh}, we identify every rooted permutation  $(\sigma,i)$ with the total order $(\Asi,\leqsi).$
Thanks to this identification, we call \emph{infinite rooted permutation} a pair $(A,\preccurlyeq)$ where $A$ is an infinite interval of integers containing 0 and $\preccurlyeq$ is a total order on $A$. We denote the set of infinite rooted permutations by $\SG^{\infty}_\bullet.$

We highlight that infinite rooted permutations can be thought of as rooted at 0. We set \break
$\tilde{\SG}_{\bullet}:=\Sr\cup\SG^{\infty}_\bullet,$
which is the set of all (finite and infinite) rooted permutations. 

We finally introduce the following $h$\textit{-restriction function around the root} defined, for every $h\in\N$, as follows
\begin{equation}
\label{rhfunct}
\begin{split}
r_h \colon\quad &\tilde{\SG}_{\bullet}\;\longrightarrow \qquad \;\SG_\bullet\\
(A&,\preccurlyeq) \mapsto \big(A\cap[-h,h],\preccurlyeq\big)\;.
\end{split}
\end{equation}
We can think of restriction functions as a notion of neighborhood around the root.
For finite rooted permutations we also have the equivalent description of the restriction functions $r_h$ in terms of consecutive patterns: if $(\sigma,i)\in\Sr$ then $r_h(\sigma,i)=(\pat_{[a,b]}(\sigma),i)$ where we take \hbox{$a=\max\{1,i-h\}$} and $b=\min\{|\sigma|,i+h\}.$

The \emph{local distance} $d$ on the set of (possibly infinite) rooted permutations $\tilde{\SG}_{\bullet}$ is defined as follows: given two rooted permutations $(A_1,\preccurlyeq_1),(A_2,\preccurlyeq_2)\in\tilde{\SG}_{\bullet},$
\begin{equation}
\label{distance}
d\big((A_1,\preccurlyeq_1),(A_2,\preccurlyeq_2)\big)=2^{-\sup\big\{h\in\N\;:\;r_h(A_1,\preccurlyeq_1)=r_h(A_2,\preccurlyeq_2)\big\}},
\end{equation}
with the classical conventions that $\sup\emptyset=0,$ $\sup\N=+\infty$ and $2^{-\infty}=0.$ 
The metric space $(\tilde{\SG}_{\bullet},d)$ is a compact space (see \cite[Theorem 2.16]{borga2018local}).

The above distance gives a notion of convergent sequences of {\em rooted} permutations. In particular, a sequence is convergent if and only if for all $h\in\N,$ the $h$- restrictions of the sequence are eventually constant.

For a sequence $\sigma_n$ of {\em unrooted} permutations,
we consider the sequence of {\em random} rooted permutations $(\sigma_n,\bm i_n)$,
where $\bm i_n$ is a uniform random index in $[|\sigma_n|]$.
We say that $\sigma_n$ converges in the Benjamini--Schramm
sense if the sequence of {\em random} rooted permutations $(\sigma_n,\bm i_n)$ 
converges in distribution for the above distance $d$.
This definition is inspired from Benjamini--Schramm convergence for graphs (see \cite{benjamini2001recurrence}).

Benjamini--Schramm convergence can be extended in two different ways 
for sequences of random permutations $(\bm \sigma_n)_{n \ge 1}$:
the \emph{annealed} and the \emph{quenched}  version of the Benjamini--Schramm convergence.
These two different versions come from the fact that there are two sources of randomness, 
one for the choice of the random permutation $\bm \sigma_n$, and one for the random root $\bm i_n$.
Intuitively, in the annealed version, the random permutation
and the random root are taken simultaneously,
while in the quenched version, the random permutation should be thought as frozen
when we take the random root.

We now give the formal definitions.
In both cases, $(\bm{\sigma}_n)_{n\in\N}$ denotes a sequence of random permutations in $\SG$ and $\bm{i}_n$ denotes a uniform index of $\bm{\sigma}_n,$ \emph{i.e.,} a uniform integer in $[|\bm{\sigma}_n|].$
\begin{definition}[Annealed version of the Benjamini--Schramm convergence]\label{defn:weakweakconv}
	We say that $(\bm{\sigma}_n)_{n\in\N}$ \emph{converges in the annealed Benjamini--Schramm sense} to a random variable $\bm{\sigma}_{\infty}$ with values in $\Sri$ if the sequence of random variables $(\bm{\sigma}_n,\bm{i}_n)_{n\in\N}$ converges in distribution to $\bm{\sigma}_{\infty}$ with respect to the local distance $d$.
	In this case we write $\bm{\sigma}_n\stackrel{aBS}{\longrightarrow}\bm{\sigma}_\infty$ instead of  $(\bm{\sigma}_n,\bm{i}_n)\stackrel{d}{\rightarrow}\bm{\sigma}_\infty.$
\end{definition}

\begin{definition}[Quenched version of the Benjamini--Schramm convergence]
	\label{strongconv}
	We say that $(\bm{\sigma}_n)_{n\in\N}$ \emph{converges in the quenched Benjamini--Schramm sense} 
	to a random measure $\bm{\nu}^\infty$ on $\Sri$
	if the sequence of conditional laws $\big(\mathcal{L}aw\big((\bm{\sigma}_n,\bm{i}_n)\big|\bm{\sigma}_n\big)\big)_{n\in\N}$
	converges in distribution to $\bm{\nu}^\infty$ with respect to the weak topology induced by the local distance $d$.
	In this case we write $\bm{\sigma}_n\stackrel{qBS}{\longrightarrow}\bm{\nu}^\infty$ instead of $\mathcal{L}aw\big((\bm{\sigma}_n,\bm{i}_n)\big|\bm{\sigma}_n\big)\stackrel{d}{\rightarrow}\bm{\nu}^\infty$.
\end{definition}

We highlight that, in the annealed version, the limiting object is a {\em random variable} with values in $\Sri$,
while for the quenched version, the limiting object $\bm{\nu}^{\infty}$ is a {\em random measure} on $\Sri$. We also note that the quenched Benjamini--Schramm convergence implies the annealed one. We have several characterizations of the two types of convergence (see \cite[Section 2.5]{borga2018local}) and we state here the one that we need for our results.

\begin{theorem}(\cite[Theorem 2.32]{borga2018local})
	\label{thm:strongbsconditions}
	For any $n\in\Z_{>0},$ let $\bm{\sigma}_n$ be a random permutation of size $n$ and $\bm{i}_n$ be a uniform random index in $[n]$, independent of $\bm{\sigma}_n.$ Then the following are equivalent:
	\begin{enumerate}
		\item there exists a random measure $\bm{\mu}^\infty$ on $\Sri$ such that   $$\bm{\sigma}_n\stackrel{qBS}{\longrightarrow}\bm{\mu}^\infty;$$ 
		\item there exists a family of non-negative real random variables $(\bm{\Gamma}^h_{\pi})_{h\in\Z_{>0},\pi\in\mathcal{S}^{2h+1}}$ such that  $$\Big(\P\big(r_h(\bm{\sigma}_n,\bm{i}_n)=(\pi,h+1)\big|\bm{\sigma}_n\big)\Big)_{h\in\Z_{>0},\pi\in\mathcal{S}^{2h+1}}\stackrel{(d)}{\to}(\bm{\Gamma}^h_{\pi})_{h\in\Z_{>0},\pi\in\mathcal{S}^{2h+1}},$$
		w.r.t. the product topology.
	\end{enumerate}
	In particular, if one of the two conditions holds (and so both) then, for all $h\in\Z_{>0}$ and $\pi\in\mathcal{S}^{2h+1}$, we have that
	\begin{equation*}
			\bm{\Gamma}^h_{\pi}\stackrel{(d)}{=}\bm{\mu}_{\infty}\Big(B\big((\pi,h+1),2^{-h}\big)\Big),
	\end{equation*}
	where $B\big((\pi,h+1),2^{-h}\big)$ denotes the ball (w.r.t. the distance introduce in  (\ref{distance})) with center $(\pi,h+1)$ and radius $2^{-h}$.
\end{theorem}

\subsection{The local limit for square permutations}

We first give some more explanations about our notation for random quantities that will be used in this section.

\subsubsection{Notation}
We will use a superscript notation on probability measure $\P$ (and on the corresponding expectation $\E$) to record the source of randomness. Specifically, given two independent random variables $\bm{X}$ and $\bm{Y}$ (with values in two spaces $E$ and $F$ respectively) and a set $A\subseteq E\times F,$ we write 
\begin{equation*}
\P^{\bm{Y}}\big((\bm{X},\bm{Y})\in A\big)\coloneqq\P\big((\bm{X},\bm{Y})\in A|\bm{X}\big),
\end{equation*}
and similarly
\begin{equation*}
\P^{\bm{X}}\big((\bm{X},\bm{Y})\in A\big)\coloneqq\P\big((\bm{X},\bm{Y})\in A|\bm{Y}\big).
\end{equation*}
Moreover, we recall the following standard relation
\begin{equation}
\label{condlaw}
\begin{split}
\P\big((\bm{X},\bm{Y})\in A\big)=\E\big[\mathds{1}_{(\bm{X},\bm{Y})\in A}\big]=\E^{\bm{X}}\Big[\E^{\bm{Y}}\big[\mathds{1}_{(\bm{X},\bm{Y})\in A}\big]\Big]&=\E^{\bm{X}}\Big[\P^{\bm{Y}}\big((\bm{X},\bm{Y})\in A\big)\Big]\\
&=\E^{\bm{Y}}\Big[\P^{\bm{X}}\big((\bm{X},\bm{Y})\in A\big)\Big].
\end{split}
\end{equation}

Finally, we denote with $\bm o_p(1)$ an unspecified
random variable $\bm{Y}_n$ of a sequence $(\bm{Y}_n)_n$ that tends to zero in probability.

\subsubsection{Construction of the limiting objects and statement of the theorem}
\label{const_lim_obj}

We start this section by introducing the candidate limiting objects for the annealed and quenched Benjamini--Schramm convergence of square permutations (see Theorem \ref{thm:local_conv}). Therefore we have to define a random infinite rooted permutation and a random measure on $\Sri$. After stating Theorem \ref{thm:local_conv}, we will also give an intuitive explanation on the construction of these limiting objects.

We start by defining the random infinite rooted permutation as a random total order $\bm{\preccurlyeq}_{\infty}$ on $\Z.$ 
We consider the set of integer numbers $\Z,$ and a labeling $\mathcal{L}\in\{+,-\}^{\Z}$ of all integers with ``$+$" or ``$-$".
We set $\mathcal{L}^+\coloneqq\{x\in\Z:x\text{ has label }``+"\}$ and $\mathcal{L}^-\coloneqq\{x\in\Z:x\text{ has label }``-"\}.$

Then we define on $\Z$ four total order $\preccurlyeq_j^\mathcal{L}$, $j\in\{1,2,3,4\}$,  saying that, for all $x,y\in\Z,$  

\[ \begin{cases} 
x\preccurlyeq_1^{\mathcal{L}} y \quad\text{ if }\quad (x<y \text{ and } x,y\in \mathcal{L}^-) \text{ or } (x<y \text{ and } x,y\in \mathcal{L}^+) \text{ or } (x\in \mathcal{L}^- \text{ and } y\in \mathcal{L}^+), \\
x\preccurlyeq_2^{\mathcal{L}} y \quad\text{ if }\quad (x>y \text{ and } x,y\in \mathcal{L}^-) \text{ or } (x<y \text{ and } x,y\in \mathcal{L}^+) \text{ or } (x\in \mathcal{L}^- \text{ and } y\in \mathcal{L}^+), \\
x\preccurlyeq_3^{\mathcal{L}} y \quad\text{ if }\quad (x<y \text{ and } x,y\in \mathcal{L}^-) \text{ or } (x>y \text{ and } x,y\in \mathcal{L}^+) \text{ or } (x\in \mathcal{L}^- \text{ and } y\in \mathcal{L}^+), \\
x\preccurlyeq_4^{\mathcal{L}} y \quad\text{ if }\quad (x>y \text{ and } x,y\in \mathcal{L}^-) \text{ or } (x>y \text{ and } x,y\in \mathcal{L}^+) \text{ or } (x\in \mathcal{L}^- \text{ and } y\in \mathcal{L}^+).
\end{cases}
\]
An example of the four constructions is given in Fig.~\ref{constr_order}.

\begin{figure}[htbp]
	\begin{center}
		\includegraphics[scale=0.8]{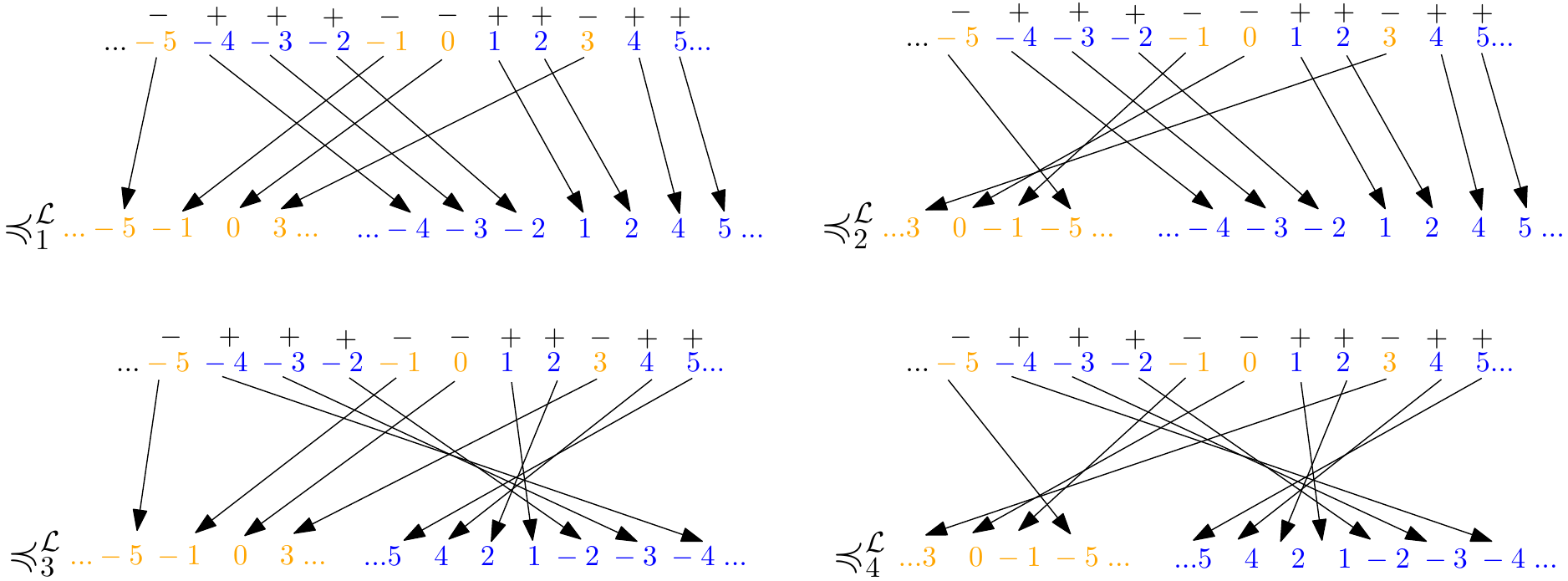}\\
		\caption{An example of the four total orders $(\Z,\preccurlyeq_j^\mathcal{L})$, for $j\in\{1,2,3,4\}$. For each of the four cases, on the top line, we see the standard total order on $\Z$ with the integers labeled by ``$-$" signs (painted in orange) and ``+" signs (painted in blue). Then, in the bottom line of each of the four cases, we move the ``$-$"-labeled numbers at the beginning of the new total order and the  ``$+$"-labeled  numbers at the end. Moreover, for $\preccurlyeq_1^{\mathcal{L}} $ we keep the relative order among integers with the same label, for $\preccurlyeq_2^{\mathcal{L}}$  we reverse the order on the ``$-$"-labeled numbers, for $\preccurlyeq_3^{\mathcal{L}}$ we reverse the order on the ``$+$"-labeled numbers and for $\preccurlyeq_4^{\mathcal{L}}$ we reverse the order on both ``$-$"-labeled and ``$+$"-labeled numbers. For each case, reading the bottom line from left to right gives the total order $\preccurlyeq_j^\mathcal{L}$ on $\Z.$ }\label{constr_order}
	\end{center}
\end{figure}

The random total order $\bm{\preccurlyeq}_{\infty}$ on $\Z$ is defined as follows. We choose a Bernoulli labeling $\bm{\mathcal{L}}$ of $\Z,$ namely, for all $x\in\Z,$ 
$$\P(x \text{ has label }``+")=\frac{1}{2}=\P(x \text{ has label }``-"),$$
independently for different values of $x.$
This random labeling determines four random total orders $\preccurlyeq_j^{\bm{\mathcal{L}}}$, $j\in\{1,2,3,4\}$. Finally, we set
\begin{equation}
\label{eq:limiting_ord}
\bm{\preccurlyeq}_{\infty}\quad\stackrel{d}{=}\quad {\preccurlyeq}_{\bm K}^{\bm{\mathcal{L}}},
\end{equation}
where $\bm K$ is a uniform random variable in $\{1,2,3,4\}$ independent of the random $\{+,-\}$-labeling.

We now also introduce the random measure on $\Sri$.
We start by defining the following function from $[0,1]^2$ to $\{1,2,3,4\}$,
\begin{equation}
\label{eq:def_map_J}
J(u,v)=\begin{cases}
1, &\text{if }  u<1/2\text{ and } u\leq v\leq 1-u, \\ 
2, &\text{if }  v< \min\{u,1-u\},  \\
3, &\text{if }  v> \max\{u,1-u\},\\
4, &\text{if }  u\geq1/2\text{ and } 1-u\leq v\leq u.
\end{cases}
\end{equation}

We consider two independent uniform random variables $\bm U,\bm V$ on the interval $[0,1]$ and we define the random probability measure $\bm{\nu}_{\infty}$ on $\Sri$ as
\begin{equation}
\label{eq:def_of_quaenched_lim}
\bm{\nu}_{\infty}=\mathcal{L}aw\big((\Z,\bm{\preccurlyeq}^{\bm{\mathcal{L}}}_{J(\bm U,\bm V) })\big|\bm U\big).
\end{equation}

\begin{theorem}
	\label{thm:local_conv}
	Let $\bm{\sigma}_n$ be a uniform random element of $Sq(n)$. Then $$\bm{\sigma}_n\stackrel{qBS}{\longrightarrow}\bm{\nu}_{\infty}\quad\text{and}\quad\bm{\sigma}_n\stackrel{aBS}{\longrightarrow}(\Z,\bm{\preccurlyeq}_{\infty}).$$
\end{theorem}	

Before proving Theorem \ref{thm:local_conv}, we try to explain the intuition behind it. In order to prove the (quenched and annealed) Benjamini--Schramm convergence for a sequence of uniform square permutations $(\bm{\sigma}_n)_n,$ we must understand, for any fixed $h\in\N,$ the behavior of the pattern induced by an $h$-restriction of $\bm{\sigma}_n$ around a uniform index $\bm{i}_n,$ denoted by $r_h(\bm{\sigma}_n,\bm i_n)$. Therefore, from now until the end of the section, we fix an integer $h\in\N.$ The pattern $r_h(\bm{\sigma}_n,\bm i_n)$ can have four ``different shapes", according to the relative position of $\bm z_0=\bm \sigma^{-1}(1),$ $\bm z_2=\bm \sigma^{-1}(n)$ and $\bm i_n.$ In particular (see also Fig.~\ref{local_example}), when $\bm i_n$ is far enough from $\bm z_0$ and $\bm z_2$ (and this will happen with high probability):
\begin{itemize}
	\item if $\bm z_0<\bm i_n<\bm z_2$ then $r_h(\bm{\sigma}_n,\bm i_n)$ is composed by two increasing sequences, one on top of the other;
	\item if $\bm i_n<\min\{\bm z_0,\bm z_2\}$ then $r_h(\bm{\sigma}_n,\bm i_n)$ is composed by two sequences, an increasing one on top of a decreasing one, \emph{i.e.,} it has a ``$\textless$"-shape;
	\item if $\bm i_n>\max\{\bm z_0,\bm z_2\}$ then $r_h(\bm{\sigma}_n,\bm i_n)$  is composed by two sequences, an decreasing one on top of an increasing one, \emph{i.e.,} it has a ``$\textgreater$"-shape;
	\item if $\bm z_0<\bm i_n<\bm z_2$ then $r_h(\bm{\sigma}_n,\bm i_n)$ is composed by two decreasing sequences, one on top of other.
\end{itemize}

\begin{figure}[htbp]
	\begin{center}
		\includegraphics[scale=0.8]{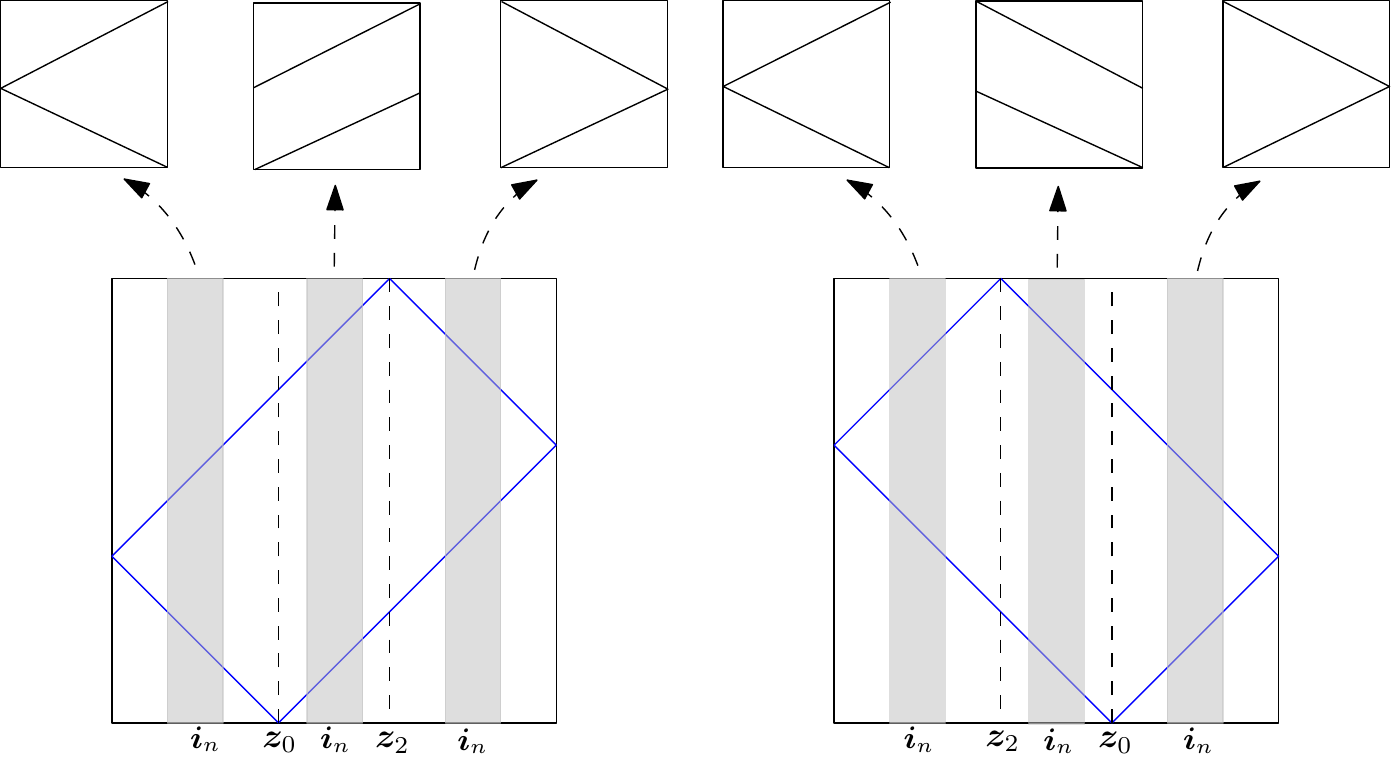}\\
		\caption{In blue the approximative shapes of two large square permutations. In the first case $\bm z_0<\bm z_2$ and in the second case $\bm z_2<\bm z_0.$ The top row shows that the possible shapes of a pattern induced by a vertical strip (six of them are highlighted in gray) around an index $\bm i_n$ is determined by the relative position of $\bm z_0,\bm z_2$ and $\bm i_n$.}\label{local_example}
	\end{center}
\end{figure}

Note that in the second and third case the fact that the two sequences are ``disjoint", \emph{i.e.,} that one is above the other, always holds for square permutations. On the other hand, the same result in the first and fourth case is true just with high probability. The goal of Section \ref{sect:technicalities} is to prove this result (see Proposition \ref{prop:technicalities}). 

The quenched and annealed limiting objects that we introduced at the beginning of this section are constructed keeping in mind these four possible cases. Specifically, for the quenched limiting object $\bm{\nu}_{\infty}$, the uniform random variables $\bm U$ and $\bm V$ involved in the definition have to be thought of as the limits of the points $\bm{i}_n$ and $\bm z_0$ respectively (after rescaling by a factor $n$).

Finally, we explain the choice of the Bernoulli labeling $\bm{\mathcal{L}}$ used in the construction of the four random total orders $\preccurlyeq_j^{\bm{\mathcal{L}}}$. It is enough to note that every point of the permutation, contained in a chosen vertical strip around $\bm{i}_n$, is in the top or bottom sequence with equiprobability and independently of the other points (this is an easy consequence of the construction presented in Section \ref{sect:inverse_projection}).

\subsubsection{The existence of the separating line.}
\label{sect:technicalities}
We introduce some more notation (\emph{cf.}\ Fig.~\ref{min_max_perm}). Given a permutation $\sigma\in Sq(n)$ such that $z_0=\sigma^{-1}(1)<\sigma^{-1}(n)=z_2,$ we define, for all $h\in\N$ and for all $i\in[z_0+h,z_2-h],$ 
\begin{equation*}
\begin{split}
&m_{i,h}^U(\sigma)\coloneqq\min\big\{\sigma(j)\big|j\in[i-h,i+h],\;\sigma_j\in\LRM(\sigma)\big\},\\
&M_{i,h}^D(\sigma)\coloneqq\max\big\{\sigma(j)\big|j\in[i-h,i+h],\;\sigma_j\in\RLm(\sigma)\big\},
\end{split}
\end{equation*}
with the conventions that $\min{\emptyset}=+\infty$ and $\max{\emptyset}=-\infty.$

We define, for a random square permutation $\bm{\sigma}$ of size $n$ and a uniform index $\bm{i}\in[n],$ the following event (conditioning on $\{\bm z_0<\bm z_2,\;\bm z_0+h\leq\bm{i}\leq\bm z_2-h\}$), for all $h\in\N,$
\begin{equation}
\label{eq:splitting_lines_events}
S_h(\bm{\sigma},\bm{i})\coloneqq\big\{m_{\bm{i},h}^U(\bm{\sigma})>M_{\bm{i},h}^D(\bm{\sigma})\big\}.
\end{equation}

This is the event that the random rooted permutation induced by the $h$-restriction $r_h(\bm{\sigma},\bm{i})$ splits into two (possibly empty) increasing subsequences with separated values, namely, the minimum of the upper subsequence is greater than the maximum of the lower subsequence. We will say, if this events hold, that ``a separating line exists" (in Fig.~\ref{min_max_perm} this separating line is dashed in orange).

\begin{figure}[htbp]
	\begin{center}
		\includegraphics[scale=0.40]{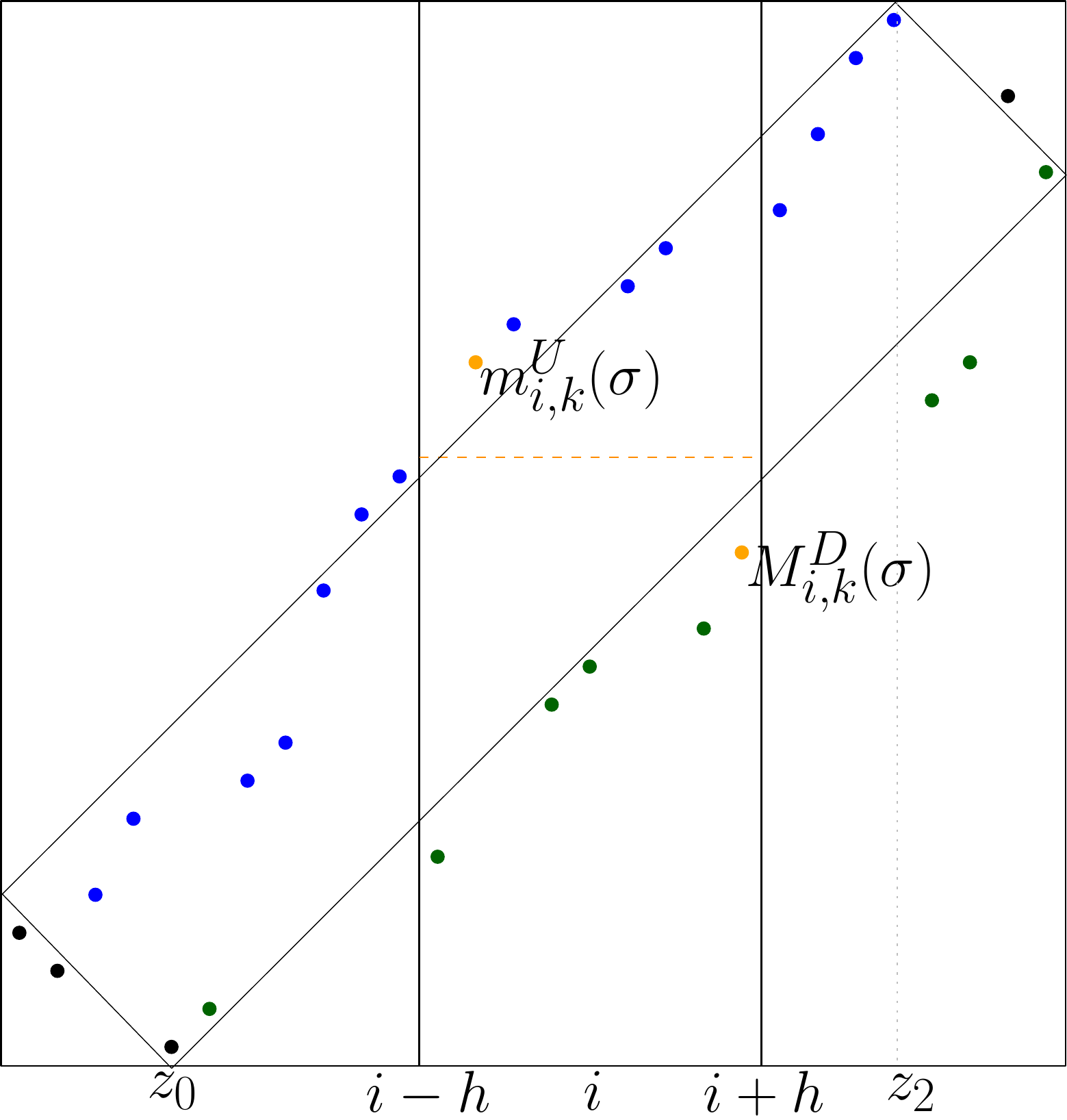}\\
		\caption{A square permutation $\sigma$.  We highlight in blue the left-to-right maxima and in green the right-to-left minima. We also painted in orange the two dots $m_{i,h}^U(\sigma)$ and $M_{i,h}^D(\sigma)$ inside the vertical strip centered in $i$ of width $2h+1.$ Moreover, the dashed orange line identifies the separating line.}\label{min_max_perm}
	\end{center}
\end{figure}

\begin{proposition}
	\label{prop:technicalities}
	For all $n\in\N,$ let $\bm{\sigma}_n$ be a uniform random square permutation of size $n$. Fix $h\in\N$ and let $\bm{i}_n$ be uniform in $[n]$ and independent of $\bm{\sigma}_n.$ Set also $E=\big\{\bm z_0<\bm z_2,\;\bm z_0+h\leq\bm{i}_n\leq\bm z_2-h\big\}.$ Then,  as $n\to\infty,$
	$$\Big|\P^{\bm{i}_n}\big(S_h(\bm{\sigma}_n,\bm{i}_n)\cap E\big)-\P^{\bm{i}_n}(E)\Big|\stackrel{P}{\longrightarrow}0.$$
\end{proposition}

\begin{proof}
	It is enough to prove the statement for a uniform random element $\bm{\sigma}_n$ of $\rho(\Omega_n)$. Then, the statement for uniform square permutations follows using Lemma \ref{square_is_rect}. We just need to show that
	\begin{equation}
	\label{eq:goal_of_theproof}
	\P(E)-\P\big(S_h(\bm{\sigma}_n,\bm{i}_n)\cap E\big)\longrightarrow 0.	
	\end{equation}
	Then, noting that almost surely
	$$\P^{\bm i_n}\big(S_h(\bm{\sigma}_n,\bm{i}_n)\cap E\big)\leq\P^{\bm i_n}(E),$$
	and using the relations 
	$\P\big(S_h(\bm{\sigma}_n,\bm{i}_n)\cap E\big)=\E^{\bm{\sigma}_n}\Big[\P^{\bm i_n}\big(S_h(\bm{\sigma}_n,\bm{i}_n)\cap E\big)\Big]$ and $\P(E)=\E^{\bm{\sigma}_n}[\P^{\bm i_n}(E)],$
	we can conclude applying Markov's inequality to $\P^{\bm i_n}(E)-\P^{\bm i_n}\big(S_h(\bm{\sigma}_n,\bm{i}_n)\cap E\big)$.
	
	We therefore study the probability
	$$\P\big(S_h(\bm{\sigma}_n,\bm{i}_n)\cap E\big)=\P\big(S_h(\bm{\sigma}_n,\bm{i}_n)\big| E\big)\cdot\P(E).$$
	We note that (see Fig.~\ref{min_max_perm}), thanks to Lemma \ref{guiding light} and Lemma \ref{lem:map_welldef}, the distance of the points in the set $\LRM(\sigma)$ (resp.\ $\RLm(\sigma)$) from the line of equation $y=x+\bm z_0$ (resp.\ $y=x-\bm z_0$) is a.s.\ bounded by $10n^{.6}.$ Therefore, conditioning on $E,$ almost surely,
	\begin{equation*}
	\begin{split}
		&m_{i,h}^U(\sigma)>\bm{i_n}-h+\bm{z_0}-20n^{.6},\\
		&M_{i,h}^D(\sigma)<\bm{i_n}+h-\bm{z_0}+20n^{.6}.
	\end{split}
	\end{equation*}
	We obtain that
	\begin{equation*}
		\P\big(S_h(\bm{\sigma}_n,\bm{i}_n)\big| E\big)\geq \P\big(\bm{z}_0>h+20n^{.6}\big| E\big)=1-\P\big(\bm{z}_0\leq h+20n^{.6}\big|\bm z_0<\bm z_2\big).
	\end{equation*}
	Using Lemma \ref{okz}, we have that a.s., $|\bm{z}_2+\bm{z}_0-n|<10n^{.6}.$ Therefore,
	\begin{multline*}
	\P\big(\bm{z}_0\leq h+20n^{.6}\big|\bm z_0<\bm z_2\big)=\P\big(\bm{z}_0\leq h+20n^{.6}\big|\bm z_0<\lfloor\tfrac{n}{2}\rfloor\big)+o(1)
	=\frac{\lfloor h+20n^{.6}\rfloor}{\lfloor\tfrac{n}{2}\rfloor}+o(1)=o(1).
	\end{multline*}
	Therefore $\P\big(S_h(\bm{\sigma}_n,\bm{i}_n)\cap E\big)=(1+o(1))\cdot\P(E)$. This implies  (\ref{eq:goal_of_theproof}) and concludes the proof.
\end{proof}

\subsubsection{The proof of the main theorem.}
\label{sect:proof_of_local_thm}

We start by introducing the following notation for a sequence $X\in\{U,D\}^n,$ and an integer $i\in[h+1,n-h],$
$$D_{i,h}(X)\coloneqq\{x\in[1,2h+1]|X_{x+i-h-1}=D\},$$ 
\emph{i.e.,} $D_{i,h}(X)$ denotes the set of indices of $D$s in $X$ in the interval $[i-h,i+h],$ shifted in the interval $[1,2h+1]$.

Recall that for a square permutation $\sigma=\rho(X,Y,z_0)$ of size $n$ obtained from the set $\Omega_n,$ we have $z_0=\sigma^{-1}(1)$ and we denote $z_2=\sigma^{-1}(n)$. We define, for all $h\in\N,$ the following map $\varphi_{h},$ for all rooted square permutations $(\sigma,i)$ such that $\sigma\in\rho(\Omega_n)$,
$$\varphi_h(\sigma,i)\coloneqq(j,D_{i,h}(X)),$$
where 
\[ 
j=\begin{cases}
1, &\text{if }\quad  z_0\leq z_2 \text{ and } z_0+h\leq i\leq z_2-h, \\ 
2, &\text{if }\quad  1+h\leq i \leq \min\{z_0,z_2\}-h, \\
3, &\text{if }\quad  \max\{z_0,z_2\}+h\leq i \leq n-h,\\
4, &\text{if }\quad z_2\leq z_0\text{ and } z_2+h\leq i \leq z_0-h,\\
\diamond, &\text{otherwise}.
\end{cases}
\]
We suggest to compare this definition with Fig.~\ref{fig:def_functions}.

Finally, for all $h\in\N,$ we define a second map, $\psi_h$, as follows: for all $j\in\{1,2,3,4\}$ and for all subsets of indexes $\mathcal{D}\subseteq[1,2h+1]$,
$$\psi_h(j,\mathcal{D})\coloneqq(\pi,h+1),$$
where $\pi$ is the unique permutation of size $2h+1$ such that, letting $\mathcal{U}=[1,2h+1]\setminus\mathcal{D}$,
\begin{equation*}
\pi(i)<\pi(j),\quad\text{for all}\quad (i,j)\in \mathcal{D}\times \mathcal{U};
\end{equation*}
and
\[ 
\begin{cases}
\text{$\pat_{\mathcal{D}}(\pi)$ and $\pat_{\mathcal{U}}(\pi)$ are increasing}, &\text{if }\quad  j=1,\\ 
\text{$\pat_{\mathcal{D}}(\pi)$ is decreasing and $\pat_{\mathcal{U}}(\pi)$ is increasing}, &\text{if }\quad j=2,\\
\text{$\pat_{\mathcal{D}}(\pi)$ is increasing and $\pat_{\mathcal{U}}(\pi)$ is decreasing}, &\text{if }\quad j=3,\\
\text{$\pat_{\mathcal{D}}(\pi)$ and $\pat_{\mathcal{U}}(\pi)$ are decreasing}, &\text{if }\quad j=4.
\end{cases}
\]
We again suggest to compare this definition with Fig.~\ref{fig:def_functions}.

\begin{figure}[htbp]
	\begin{center}
		\includegraphics[scale=1]{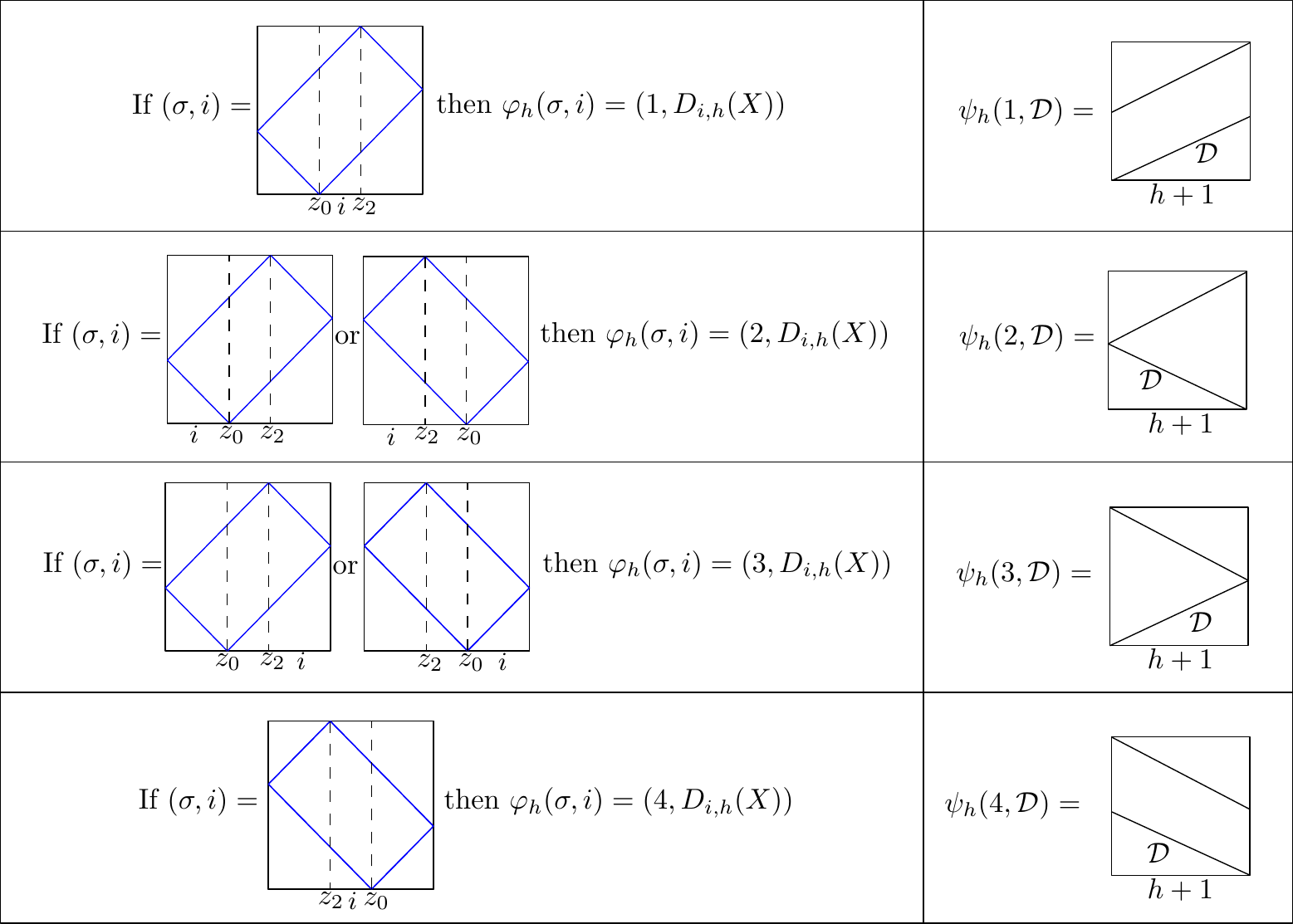}\\
		\caption{A table that summarizes all the different cases in the definitions of the functions $\varphi_h$ and $\psi_h$.}\label{fig:def_functions}
	\end{center}
\end{figure}

\medskip

Note that if $\varphi_h(\sigma,i)=(j,D_{i,h}(X))$ and $j\in\{2,3\}$ then trivially,
\begin{equation}
\label{eq:key_trick}
\psi_h\big(\varphi_h(\sigma,i)\big)=r_h(\sigma,i).
\end{equation}
On the other hand, if $j=1$ (or $j=4$) we have to take a bit of care. Indeed, in full generality,  (\ref{eq:key_trick}) is false (see for instance the example presented in Fig.~\ref{fail2} on page~\pageref{fail2}). Nevertheless, if $S_h(\sigma,i)$ holds (see  (\ref{eq:splitting_lines_events}) for the notation), \emph{i.e.,} a separating line exists, then   (\ref{eq:key_trick}) is true.
 
We now consider a uniform square permutation $\bm \sigma_n$ in $\rho(\Omega_n)$ and a uniform index $\bm i_n$ in $[n].$ Noting that 
\begin{equation*}
\P^{\bm i_n}\big(\varphi_h(\bm\sigma_n,\bm i_n)\in(\diamond,\cdot))=\bm o_p(1),
\end{equation*}
and using Proposition \ref{prop:technicalities} (and the analogous result for the symmetric case when $\bm z_2<\bm z_0$), we can conclude that
\begin{equation}
\label{eq:constrnotfail}
\P^{\bm i_n}\big(\psi_h\big(\varphi_h(\bm\sigma_n,\bm i_n)\big)\neq r_h(\bm\sigma_n,\bm i_n)\big)=\bm o_p(1).
\end{equation}

\medskip

For later convenience we also present a generalization of the  (\ref{eq:key_trick}) for the random total orders $(\Z,\bm{\preccurlyeq}^{\bm{\mathcal{L}}}_{j})$. We define for a labeling $\mathcal{L}\in\{+,-\}^{\Z}$ and an integer $h\in\N,$
$$\mathcal{D}_{h}(\mathcal{L})\coloneqq\{x\in[1,2h+1]|\mathcal{L}_{x-h-1}=``-"\},$$
\emph{i.e.,} $\mathcal{D}_{h}(\mathcal{L})$ denotes the set of indices of ``$-$" in $\mathcal{L}$ in the interval $[-h,h],$ shifted in the interval $[1,2h+1]$. It trivially holds that 
\begin{equation}
\label{eq:maponinfiniteorder}
r_h(\Z,\bm{\preccurlyeq}^{\bm{\mathcal{L}}}_j)=\psi_h(j,\mathcal{D}_{h}(\mathcal{L})).
\end{equation}

\medskip

Before proving Theorem \ref{thm:local_conv}, we state a result regarding the number of occurrences of a fixed pattern in a uniform random sequence $\bm X\in\{U,D\}^n.$ Given a pattern $w\in\{U,D\}^k,$ for some $k<n,$ we denote with $\coc(w,\bm X)$ the number of consecutive occurrences of $w$ in $\bm X.$ 

\begin{lemma}
\label{lem:flajole}
With the notation above,
\begin{equation*}
	\frac{\E[\coc(w,\bm X)]}{n}\to\big(\tfrac{1}{2}\big)^{k}\quad\text{and}\quad \frac{\emph{Var}[\coc(w,\bm X)]}{n^2}\to 0.
\end{equation*}	
\end{lemma}

\begin{proof}
	This follows from \cite[Proposition IX.10.]{FlSe09}.	
\end{proof}

We can now prove Theorem \ref{thm:local_conv}.

\begin{proof}[Proof of Theorem \ref{thm:local_conv}]
	It is enough to prove the statement for a uniform random element of $\rho(\Omega_n)$. Then, the statement for uniform square permutations follows using Lemma \ref{square_is_rect}.
	
	We recall that a uniform random element of $\rho(\Omega_n)$ can be sampled as follows.
	Let $\zz_0$ be an integer chosen uniformly from $(n^{.9},n-n^{.9})$ and let $(\bmX,\bmY)$ be uniform in $\{U,D\}^n\times \{L,R\}^n$ conditioned to satisfy the Petrov conditions and to have $\bmY_{\zz_0}=D$.
	Under these assumptions $(\bmX,\bmY,\zz_0)$ is a uniform random element of $\Omega_n$ and consequently $\bm{\sigma}_n\coloneqq\rho((\bmX,\bmY,\zz_0))$ is a uniform random element of $\rho(\Omega_n)$. We also set $\bm{z}_2=\bm{z}_2(n)=\bm{\sigma}_n^{-1}(n)$. 
	
	\medskip

	We denote with $\bm{i}_n$ a uniform random index in $[n]$, independent of $\bm{\sigma}_n.$ We also recall that we denote with $B\big((\pi,h+1),2^{-h}\big),$ for all $\pi\in\mathcal{S}^{2h+1},$ the ball (w.r.t. the distance introduce in  (\ref{distance})) with center $(\pi,h+1)$ and radius $2^{-h}$. Thanks to Theorem \ref{thm:strongbsconditions} in order to prove the quenched local convergence, it is enough to check that 
	\begin{equation}
	\label{eq:goaloftheproof}
	\Big(\P^{{\bm{i}}_n}\big(r_h(\bm{\sigma}_n,\bm{i}_n)=(\pi,h+1)\big)\Big)_{h\in\Z_{>0},\pi\in\mathcal{S}^{2h+1}}\stackrel{d}{\to}\Big(\bm{\nu}_{\infty}\Big(B\big((\pi,h+1),2^{-h}\big)\Big)\Big)_{h\in\Z_{>0},\pi\in\mathcal{S}^{2h+1}},
	\end{equation}
	w.r.t. the product topology.

Fix $h\in\N.$ Using  (\ref{eq:constrnotfail}), for all pattern $\pi\in\mathcal{S}^{2h+1}$,
	\begin{equation}
	\label{eq:big_prob_split}
		\P^{{\bm{i}}_n}\big(r_h(\bm{\sigma}_n,\bm{i}_n)=(\pi,h+1)\big)=\sum_{(j,\mathcal{D})\in\psi^{-1}(\pi,h+1)}\P^{{\bm{i}}_n}\big(\varphi_h(\bm{\sigma}_n,\bm{i}_n)=(j,\mathcal{D})\big)
		+\;\bm o_p(1).
	\end{equation}
	We analyze the term in  (\ref{eq:big_prob_split}) that involves $(1,\mathcal{D})$. The other three cases are similar. 
	
	\medskip
	
	Note that, by definition of $\varphi_h$, the event $\{\varphi_h(\bm{\sigma}_n,\bm{i}_n)=(1,\mathcal{D})\}$ holds if and only if $\bm z_0\leq \bm z_2,$ $\bm z_0+h\leq \bm i_n\leq \bm z_2-h$ and $D_{\bm{i}_n,h}(\bm X)=\mathcal{D}.$ Therefore, recalling that $E=\{\bm{z}_0\leq\bm{z}_2,\; \bm{z}_0+h\leq\bm{i}_n\leq\bm{z}_2-h\},$ we have
	\begin{equation*}
	\P^{{\bm{i}}_n}\big(\varphi_h(\bm{\sigma}_n,\bm{i}_n)=(1,\mathcal{D})\big)=\P\big(E\cap\{D_{\bm{i}_n,h}(\bm X)=\mathcal{D}\} \big|\bm z_0,\bm X\big).
	\end{equation*}
	Using Lemma \ref{okz}, we have that a.s., $|\bm{z}_2+\bm{z}_0-n|<10n^{.6}.$ Therefore the above conditional probability can be rewritten as
	\begin{equation}
	\P\big(E^*\cap\{D_{\bm{i}_n,h}(\bm X)=\mathcal{D}\} \big|\bm z_0,\bm X\big)+\bm o_p(1),
	\end{equation}
	where $E^*=\{\bm{z}_0\leq n-\bm{z}_0,\; \bm{z}_0+h\leq\bm{i}_n\leq n-\bm{z}_0-h\}.$
	Noting that, conditioning on $E^*$, $\bm{i}_n$ is distributed like a uniform random variable $\bm{i}^*_n$ in the interval $[\bm z_0+h,n-\bm z_0-h],$ we have that
	\begin{equation*}
	\P^{{\bm{i}}_n}\big(\varphi_h(\bm{\sigma}_n,\bm{i}_n)=(1,\mathcal{D})\big)=\P\big(E^*\big|\bm z_0\big)\cdot\P\big(D_{\bm{i}^*_n,h}(\bm X)=\mathcal{D}\big|\bm z_0,\bm X\big)+\bm o_p(1).
	\end{equation*}
	Moreover, noting that 
	\begin{equation*}
	\P\big(E^*\big|\bm z_0\big)=\frac{n-2\bm z_0-2h-1}{n}\mathds{1}_{\{\bm z_0<n-\bm z_0\}},
	\end{equation*}
	and using that $\frac{\bm z_0}{n}\stackrel{d}{\longrightarrow}\bm U,$ where $\bm U$ is a uniform random variable in $[0,1],$ we obtain
	\begin{equation}
	\label{eq:first_step}
	\P\big(E^*\big|\bm z_0\big)\stackrel{d}{\longrightarrow}(1-2\bm U)\mathds{1}_{\{\bm U<\tfrac{1}{2}\}}.
	\end{equation}
	We now prove using the Second moment method that 
	\begin{equation}
	\label{eq:goal_sec_mom}
	\P\big(\{D_{\bm{i}^*_n,h}(\bm X)=\mathcal{D}\}\big|\bm z_0,\bm X\big)\stackrel{P}{\longrightarrow}2^{-(2h+1)}.
	\end{equation}
	Let $w^\mathcal{D}\in\{U,D\}^{2h+1}$ be the unique sequence such that $w^\mathcal{D}_i=D$ if and only if $i\in\mathcal{D}.$ We can rewrite $\P\big(\{D_{\bm{i}^*_n,h}(\bm X)=\mathcal{D}\}\big|\bm z_0,\bm X\big)$ as
	\begin{equation*}
	\frac{\coc(w^\mathcal{D},\bm X_{|_{[\bm z_0,n-\bm z_0]}})}{\max\{n-2\bm z_0-2h-1,1\}},
	\end{equation*}
	where $\bm X_{|_{[\bm z_0,n-\bm z_0]}}$ denotes the sequence $\bm X$ restricted to the set of indexes between $\bm z_0$ and $n-\bm z_0.$
	
	By Chebyschev's inequality, for any fixed $\varepsilon>0,$
	\begin{multline*}
	\P^{(\bm z_0,\bm X)}\Big(\Big|	\P\big(\{D_{\bm{i}^*_n,h}(\bm X)=\mathcal{D}\}\big|\bm z_0,\bm X\big)-\E^{(\bm z_0,\bm X)}\big[\P\big(\{D_{\bm{i}^*_n,h}(\bm X)=\mathcal{D}\}\big|\bm z_0,\bm X\big)\big]\Big|\geq\varepsilon\Big)\\
	\leq\frac{1}{\varepsilon^2}\cdot\text{Var}^{(\bm z_0,\bm X)}\big(\P\big(\{D_{\bm{i}^*_n,h}(\bm X)=\mathcal{D}\}\big|\bm z_0,\bm X\big))\big).
	\end{multline*}
	Thanks to Lemma \ref{omega_size}, we can assume that the random sequence $\bm X$ is a uniform sequence in $\{U,D\}^n,$ although the sequence $\bm X$ is conditioned to satisfy the Petrov conditions. Therefore, noting that $\bm z_0$ is uniformly distributed and independent of $\bm X,$ and using Lemma \ref{lem:flajole}, the right-hand side of the previous equation tends to zero and  $\E^{(\bm z_0,\bm X)}\big[\P\big(\{D_{\bm{i}^*_n,h}(\bm X)=\mathcal{D}\}\big|\bm z_0,\bm X\big)\big]$ tends to $2^{-(2h+1)}$, proving  (\ref{eq:goal_sec_mom}). 
	
	Using the results in (\ref{eq:first_step}) and (\ref{eq:goal_sec_mom}) we can conclude that
	\begin{equation}
	\label{eq:first_key_res}
	\P^{{\bm{i}}_n}\big(\{\varphi_h(\bm{\sigma}_n,\bm{i}_n)=(1,\mathcal{D})\}\big)\stackrel{d}{\longrightarrow}2^{-(2h+1)}\cdot(1-2\bm U)\mathds{1}_{\{\bm U<\tfrac{1}{2}\}}.
	\end{equation}
	
	\medskip
	
	From (\ref{eq:big_prob_split}) and (\ref{eq:first_key_res}), and the natural extension of the second equation to the other three similar cases, we obtain
	\begin{multline}
	\label{eq:for_coroll}
	\P^{{\bm{i}}_n}\big(r_h(\bm{\sigma}_n,\bm{i}_n)=(\pi,h+1)\big)\stackrel{d}{\longrightarrow}\\
	2^{-(2h+1)}\Big(e_1(\pi)\cdot(1-2\bm U)\mathds{1}_{\{\bm U<\tfrac{1}{2}\}}+e_2(\pi)\cdot\min\{\bm U,1-\bm U\}\\
	+e_3(\pi)\cdot\max\{\bm U,1-\bm U\}+e_4(\pi)\cdot(2\bm U-1)\mathds{1}_{\{\bm U>\tfrac{1}{2}\}}\Big),
	\end{multline}
	where $e_\ell(\pi)=\big|\{\mathcal{D}|\psi_h(\ell,\mathcal{D})=(\pi,h+1)\}\big|,$ for all $\ell\in\{1,2,3,4\}.$

	Now, using the definition of $\bm{\nu}_{\infty}$ given in  (\ref{eq:def_of_quaenched_lim}), we have that
	\begin{multline*}
	\bm{\nu}_{\infty}\Big(B\big((\pi,h+1),2^{-h}\big)\Big)=\P\left(r_h(\Z,\bm{\preccurlyeq}^{\bm{\mathcal{L}}}_{J(\bm U,\bm V) })=(\pi,h+1)\Big|\bm U\right)\\
	=\P\left(\psi_h(J(\bm U,\bm V),\mathcal{D}_{h}(\bm{\mathcal{L}}))=(\pi,h+1)\Big|\bm U\right),
	\end{multline*}
	where in the last equality we used the relation in  (\ref{eq:maponinfiniteorder}). From the definition of the map $J$ given in (\ref{eq:def_map_J}) we conclude that
	\begin{multline*}
	\bm{\nu}_{\infty}\Big(B\big((\pi,h+1),2^{-h}\big)\Big)=2^{-(2h+1)}\Big(e_1(\pi)\cdot(1-2\bm U)\mathds{1}_{\{\bm U<\tfrac{1}{2}\}}+e_2(\pi)\cdot\min\{\bm U,1-\bm U\}\\
	+e_3(\pi)\cdot\max\{\bm U,1-\bm U\}+e_4(\pi)\cdot(2\bm U-1)\mathds{1}_{\{\bm U>\tfrac{1}{2}\}}\Big).
	\end{multline*}
	Therefore we can conclude that  (\ref{eq:goaloftheproof}) holds component-wise, for all $h\in\N$ and for all pattern $\pi\in\mathcal{S}^{2h+1}.$ Finally, noting that all the components of the vector in   (\ref{eq:goaloftheproof}) depends on the same realization of $\bm z_0$ and $\bm X$ we can deduce that  (\ref{eq:goaloftheproof}) holds with respect to the product topology.
	
	Using \cite[Proposition 2.35]{borga2018local}, the annealed statement is an easy corollary of the quenched convergence, noting that
	\begin{equation}
	\label{eq:expectation_equalities}
	\E[(1-2\bm U)\mathds{1}_{\{\bm U<\tfrac{1}{2}\}}]=\E[\min\{\bm U,1-\bm U\}]=\E[\max\{\bm U,1-\bm U\}]=\E[(2\bm U-1)\mathds{1}_{\{\bm U>\tfrac{1}{2}\}}]=\frac{1}{4}.
	\end{equation}
	This concludes the proof of the theorem.
\end{proof}

We conclude this section computing explicitly the limit of $\P\big(r_h(\bm{\sigma}_n,\bm{i}_n)=(\pi,h+1)\big),$ for all $h\in\N$ and all pattern $\pi\in\mathcal{S}^{2h+1}$. 

Before doing that, we introduce six subfamilies of permutations that contain all the possible patterns obtained through the maps $\psi_h$, for all $h\in\N$. We consider only permutations of size strictly greater than $2$. An example for each family is given in Fig.~\ref{example_families_perm}.
\begin{itemize}
	\item $\mathcal{A}_1$ is the family of non-monotone\footnote{We recall that a \emph{monotone} permutation $\sigma$ of size $n$ is either $\sigma=12,\dots,n$ or $\sigma=n,\dots,21.$} permutations $\pi$ of size $2h+1,$ for some $h\in\N,$ which can be written as $(\pi,h+1)=\psi_h(1,\mathcal{D})$ for some $\mathcal{D}.$ Note that $\mathcal{D}$ is uniquely determined by $\pi.$ 
	\item $\mathcal{A}_2$ is the family of permutations $\pi$ of size $2h+1,$ for some $h\in\N,$ which can be written as $(\pi,h+1)=\psi_h(2,\mathcal{D})$ for some $\mathcal{D}.$ Note that $\mathcal{D}$ is not uniquely determined by $\pi.$ More precisely, for all permutation $\pi\in\mathcal{A}_2,$ the index $1$ can be included either in $\mathcal{D}$ or not. Therefore, for each $\pi\in\mathcal{A}_2$ there are two possible choices for $\mathcal{D}$.
	\item $\mathcal{A}_3$ is the family of permutations $\pi$ of size $2h+1,$ for some $h\in\N,$ which can be written as $(\pi,h+1)=\psi_h(3,\mathcal{D})$ for some $\mathcal{D}.$ Remarks similar to the ones done for $\mathcal{A}_2$  hold also for this family.
	\item $\mathcal{A}_4$ is the family of non-monotone permutations $\pi$ of size $2h+1,$ for some $h\in\N,$ which can be written as $(\pi,h+1)=\psi_h(4,\mathcal{D})$ for some $\mathcal{D}.$ Note that $\mathcal{D}$ is uniquely determined by $\pi.$
	\item $\mathcal{A}_5$ is the family of increasing permutations $\pi$ of size $2h+1,$ for some $h\in\N.$ Note that these permutations can be written as $(\pi,h+1)=\psi_h(1,\mathcal{D})$ for some $\mathcal{D}.$ More precisely $\mathcal{D}\in\{\emptyset\}\cup\bigcup_{k\in[1,|\pi|]}\{[1,k]\},$ and so there are $|\pi|+1$ possible choices for $\mathcal{D}$.
	\item $\mathcal{A}_6$ is the family of decreasing permutations $\pi$ of size $2h+1,$ for some $h\in\N.$ Note that these permutations can be written as $(\pi,h+1)=\psi_h(4,\mathcal{D})$ for some $\mathcal{D}.$ More precisely $\mathcal{D}\in\{\emptyset\}\cup\bigcup_{k\in[1,|\pi|]}\{[k,|\pi|]\},$ and so there are $|\pi|+1$ possible choices for $\mathcal{D}$.
\end{itemize}

We note that these families are not disjoint. For example $\mathcal{A}_5\subseteq\mathcal{A}_2,$ and if $\pi\in\mathcal{A}_2$ has one of the corresponding $\mathcal{D}$ of cardinality one then $\pi\in\mathcal{A}_1$. The definitions of the families $\mathcal{A}_i$ will be clearer in the proof of Corollary \ref{cor_notobvious}. 

Given the diagram of a permutation in one of the six families, we call \emph{separating line}, a horizontal line in the diagram which splits the permutation into two monotone subsequences (we have two possible choices for every permutation in $\mathcal{A}_2$ and $\mathcal{A}_3$ and $|\pi|+1$ possible choices for every permutation $\pi$ in $\mathcal{A}_5$ and $\mathcal{A}_6$). 
\begin{figure}[htbp]
	\begin{center}
		\includegraphics[scale=.60]{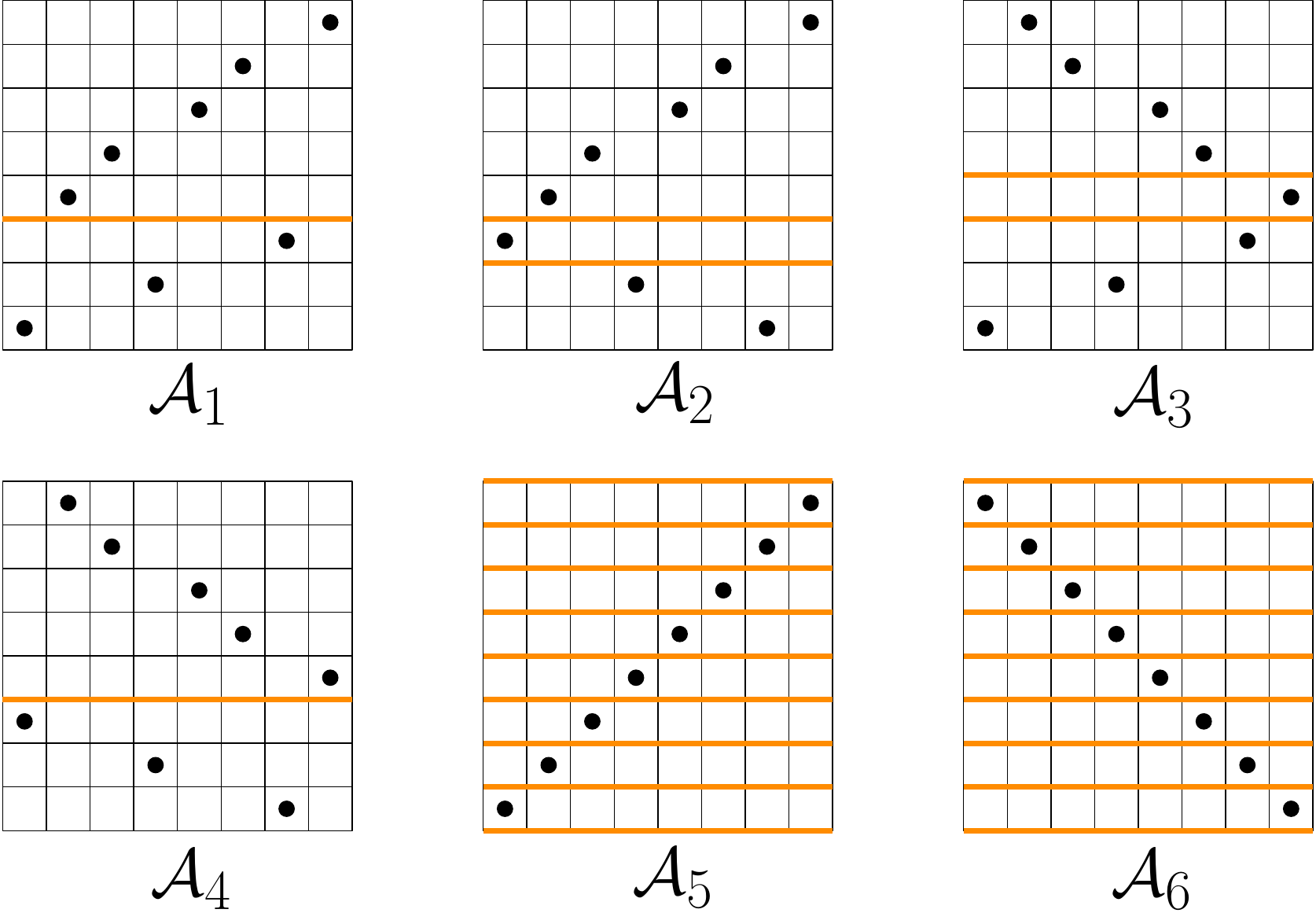}\\
		\caption{Diagram of six permutations in $\mathcal{A}_1,\mathcal{A}_2,\mathcal{A}_3,\mathcal{A}_4,\mathcal{A}_5$ and $\mathcal{A}_6$ with the potential separating lines highlighted in orange.}\label{example_families_perm}
	\end{center}
\end{figure}

We can now prove the following easy consequence of the proof of Theorem \ref{thm:local_conv}.
  
\begin{corollary}
	\label{cor_notobvious}
	Let $\bm{\sigma}_n$ be a uniform random element of $Sq(n)$. For all $h\in\N,$ and for all $\pi\in\mathcal{S}^{2h+1},$
	$$\P\big(r_h(\bm{\sigma}_n,\bm{i}_n)=(\pi,h+1)\big)\to p(\pi),$$
	where $$p(\pi)=2^{-|\pi|-2}\Big(\mathds{1}_{\mathcal{A}_1}(\pi)+2\mathds{1}_{\mathcal{A}_2}(\pi)+2\mathds{1}_{\mathcal{A}_3}(\pi)+\mathds{1}_{\mathcal{A}_4}(\pi)+(|\pi|+1)\mathds{1}_{\mathcal{A}_5\cup \mathcal{A}_6}(\pi)\Big).$$
\end{corollary}

\begin{proof}
	From (\ref{eq:for_coroll}) and (\ref{eq:expectation_equalities}) we have that
	$$\P\big(r_h(\bm{\sigma}_n,\bm{i}_n)=(\pi,h+1)\big)\to2^{-|\pi|-2}\cdot\big(e_1(\pi)+e_2(\pi)+e_3(\pi)+e_4(\pi)\big).$$
	It remains to compute the values $e_1(\pi),e_2(\pi),e_3(\pi)$ and $e_4(\pi).$ We start with $e_1(\pi)$. Note that, for every subset of indexes $\mathcal{D}\subseteq[1,2h+1],$ then $\psi_h(1,\mathcal{D})\in\mathcal{A}_1\cup\mathcal{A}_5.$
	Moreover, the cardinality $e_1(\pi)=\big|\{\mathcal{D}|\psi_h(1,\mathcal{D})=(\pi,h+1)\}\big|,$ for all $\pi\in\mathcal{A}_1\cup\mathcal{A}_5,$ is determined by the number of possible choices for the separating lines of $\pi.$ Therefore, using the discussion done during the definitions of the families $\mathcal{A}_1$ and $\mathcal{A}_5$, we conclude that
	$$e_1(\pi)=\mathds{1}_{\mathcal{A}_1}(\pi)+(|\pi|+1)\mathds{1}_{\mathcal{A}_5}(\pi).$$
	
	For $e_2(\pi)$, note that, for every subset of indexes $\mathcal{D}\subseteq[1,2h+1],$ then $\psi_h(2,\mathcal{D})\in\mathcal{A}_2.$
	Moreover, the cardinality $e_2(\pi)=\big|\{\mathcal{D}|\psi_h(2,\mathcal{D})=(\pi,h+1)\}\big|,$ for all $\pi\in\mathcal{A}_2,$ is again determined by the number of possible choices for the separating lines of $\pi,$ that is 2. Therefore, we conclude that
	$$e_2(\pi)=2\mathds{1}_{\mathcal{A}_2}(\pi).$$

	With similar arguments, we obtain that
	\begin{equation*}
		e_3(\pi)=2\mathds{1}_{\mathcal{A}_3}(\pi),\quad\text{and}\quad e_4(\pi)=\mathds{1}_{\mathcal{A}_4}(\pi)+(|\pi|+1)\mathds{1}_{\mathcal{A}_6}(\pi).\qedhere
	\end{equation*}
\end{proof}

\section{Proportion of pattern and consecutive pattern occurrences}
\label{sect:pattern_conv}
In this final section we directly deduce from Theorems \ref{thm:perm_conv} and \ref{thm:local_conv}, that the proportion of occurrences (resp.\ consecutive occurrences) of any given pattern in a uniform random square permutation converges to a \emph{random} quantity.

There are various important characterizations for the permuton convergence (see \cite[Thm. 2.5]{bassino2017universal}) and the quenched version of the Benjamini--Schramm convergence (see \cite[Theorem 2.32]{borga2018local}), some involving convergence of pattern proportions. 

In particular, for permuton limits, if $\bm{\sigma}_n$ is a random permutation of size $n$, then the following statements are equivalent:
\begin{enumerate}
	\item There exists a random permuton $\bm{\mu}$ such that $\mu_{\bm{\sigma}_n} \stackrel{d}{\to} \bm{\mu}$;
	\item The random infinite vector $(\occ(\pi,\bm{\sigma}_n))_{\pi\in\SG}$ converges in distribution in the product
	topology to some random infinite vector $(\bm\Lambda_{\pi})_{\pi\in\SG}$;
\end{enumerate} 
and whenever the assertions (1) or (2) are verified, we have for every pattern $\pi\in\SG_k$,
\[ \bm\Lambda_\pi = \P(\Perm_{k}(\bm\mu) = \pi|\bm\mu),\] 
where $\Perm_{k}(\bm\mu)$ denotes the unique permutation induced by $k$ independent points in $[0,1]^{2}$ with common distribution $\bm \mu$ conditionally\footnote{This is possible by considering the new probability space described in \cite[Section 2.1]{bassino2017universal}} on $\bm \mu$.

On the other hand, for local limits, the following statements are equivalent:
\begin{enumerate}
\item There exist a random measure $\bm{\nu}$ on the set of all rooted permutations $\Sri$ such that $\bm{\sigma}_n\stackrel{qBS}{\longrightarrow}\bm{\nu};$
\item The random infinite vector $\big(\widetilde{\coc}(\pi,\bm{\sigma}_n)\big)_{\pi\in\SG}$ converges in distribution in the product
topology to some random infinite vector $(\bm\Delta_{\pi})_{\pi\in\SG}$;
\end{enumerate}
and whenever the assertions (1) or (2) are verified, we have for every pattern $\pi\in\SG_{2h+1}$, $h\in\N$,
\[ \bm\Delta_\pi = \bm{\nu}\Big(B\big((\pi,h+1),2^{-h}\big)\Big),\]
where we recall that $B\big((\pi,h+1),2^{-h}\big)$ denotes the ball with center $(\pi,h+1)$ and radius $2^{-h}.$
We point out that the expressions for patterns of even size can be deduced from the ones of odd size using the relation $\bm\Delta_\pi =\sum_{m=1}^{|\pi|+1}\bm\Delta_{\pi^{*m}},$ where $\pi^{*m}$ is the permutation obtained by adding an additional final element in the diagram of $\pi$ immediately below the $m$-th row.

\medskip 

We recall that $\mu^{\zz}$ denote the permuton limit and $\bm{\nu}_{\infty}$ the quenched local limit for a sequence of uniform square permutations. 
\begin{corollary}
	\label{pat_conv}
	Let $\bm{\sigma}_n$ be a uniform random element of $Sq(n)$. We set 
	$$\bm\Lambda_\pi = \P(\Perm_{|\pi|}(\mu^{\zz}) = \pi|\mu^{\zz}),\quad\text{for all}\quad\pi\in\mathcal{S},$$ 
	and
	\[ \bm\Delta_\pi = \bm{\nu}_\infty\Big(B\big((\pi,h+1),2^{-h}\big)\Big),\quad\text{for all}\quad h\in\N, \pi\in\mathcal{S}^{2h+1}.\]
	The following convergences hold w.r.t. the product topology:
	$$(\occ(\pi,\bm{\sigma}_n))_{\pi\in\SG}\stackrel{d}{\rightarrow}(\bm\Lambda_{\pi})_{\pi\in\SG}\quad\text{and}\quad\big(\widetilde{\coc}(\pi,\bm{\sigma}_n)\big)_{\pi\in\SG}\stackrel{d}{\rightarrow}(\bm\Delta_{\pi})_{\pi\in\SG}.$$  
\end{corollary}

\section*{Acknowledgements}
The authors are very grateful to Mathilde Bouvel and Valentin F\'eray for the constant and stimulating discussions and suggestions. We also would like to thank Enrica Duchi for introducing us to the model of square permutations.
Finally, we thank the anonymous referee for all his/her precious and useful comments.

This work was completed with the support of the SNF grant number $200021\_172536$, ``Several aspects of the study of non-uniform random permutations" and the ERC starting grant 680275 MALIG.

\bibliographystyle{plain}
\bibliography{pattern}

\end{document}